\documentclass[11pt]{article}
\usepackage[T1]{fontenc}
\usepackage[utf8]{inputenc}
\usepackage{lmodern}
\usepackage[margin=1in]{geometry}
\usepackage{microtype}
\usepackage{amsmath,amssymb,amsthm,mathtools,mathrsfs,bm}
\usepackage{paralist}
\usepackage{comment}
\usepackage{mdframed}

\usepackage[dvipsnames]{xcolor}
\usepackage[colorlinks]{hyperref}
\hypersetup{
colorlinks=true,
    linkcolor=red,
    citecolor=OliveGreen,
    filecolor=black,
    urlcolor=black,
}

\allowdisplaybreaks[3]
\numberwithin{equation}{section}

\let\le\leqslant

\let\ge\geqslant

\usepackage{xfrac}
\usepackage{mathrsfs}
\usepackage{eucal}
\usepackage{dsfont}

\newcommand{\e}{{\mathrm{e}}}

\newtheorem{theorem}{Theorem}
\newtheorem{proposition}[theorem]{Proposition}
\newtheorem{lemma}[theorem]{Lemma}

\newtheorem{definition}{Definition}
\theoremstyle{remark}

\newcommand{\proofstep}[1]{$\textbf{#1}^{\circ}$}

\newcommand{\wt}{\widetilde}

\newcommand{\R}{\mathbb{R}}

\newcommand{\ip}[2]{\left\langle #1,#2\right\rangle}
\newcommand{\norm}[1]{\left\lVert #1\right\rVert}
\newcommand{\ones}{\mathsf{1}_d}
\newcommand{\veps}{\varepsilon}

\newcommand{\inner}[2]{\langle #1,#2\rangle}
\newcommand{\abs}[1]{\lvert #1\rvert}
\newcommand{\ri}{\operatorname{ri}}

\newcommand{\supp}{\operatorname{supp}}
\newcommand{\conv}{\operatorname{conv}}
\newcommand{\Env}{S}
\newcommand{\Ent}{\textup{\textsf{Ent}}}
\newcommand{\Diag}{\textup{Diag}}
\newcommand{\Pin}{\mathsf{\Pi}}
\newcommand{\Orc}{\mathscr{O}}
\newcommand{\Mtd}{\mathscr{M}}
\newcommand{\FOM}{\textsf{FOM}}
\let\Div\relax\newcommand{\Div}{D_h}
\let\Dh\Div
\newcommand{\Fh}{\Ent_d}
\newcommand{\Vh}{\textsf{Lip}_d}
\newcommand{\Simp}{\Delta_d}

\newcommand{\cI}{\mathcal{I}}
\newcommand{\ProjTan}{\Pi_{\Tan}}

\newcommand{\bOmega}{\boldsymbol{\Omega}}
\newcommand{\bPhi}{\boldsymbol{\Phi}}

\newcommand{\Herm}{\mathbf{H}^d}
\newcommand{\Dens}{\boldsymbol{\Delta}_d}
\newcommand{\Qd}{\mathsf{Ent}_d^\mathsf{H}}

\newcommand{\qre}[2]{D_{h \hspace{1pt} \circ \lambda}(#1,#2)}
\newcommand{\Tan}{\Theta_d}
\newcommand{\spanof}[1]{\operatorname{span}\{#1\}}

\newcommand{\Risk}{\mathsf{Risk}}
\newcommand{\prox}{\operatorname{prox}}
\DeclareMathOperator*{\argmin}{\arg\!\min}

\DeclareMathOperator*{\Argmin}{\operatorname{Arg}\!\min}

\DeclareMathOperator{\tr}{tr}
\DeclareMathOperator{\diag}{diag}

\newcommand{\bnabla}{\nabla}

\newcommand{\half}{\frac{1}{2}}

\title{Entropy-Smooth Convex Optimization Cannot Be Accelerated}
\author{
Jacob M.~Aguirre\thanks{Georgia Institute of Technology, H.~Milton Stewart School of Industrial and Systems Engineering (ISyE), Atlanta, USA. 
Email: \texttt{aguirre@gatech.edu}.}
\and 
Dmitrii M.~Ostrovskii\thanks{Georgia Institute of Technology, School of Mathematics \& H.~Milton Stewart School of Industrial and Systems Engineering (ISyE), Atlanta, USA. Email: \texttt{ostrov@gatech.edu}.}
}
\date{}

\begin{document}
\maketitle

\begin{abstract}
We prove an $\Omega(L/T)$ lower bound for the convergence rate of minimization in the class of functions that are convex and~$L$-smooth relative to negative entropy on the standard~$d$-simplex, valid for every first-order method when~$\smash{d = \Omega(T^2)}$. 
In particular, this shows that mirror descent is optimal up to a logarithmic factor in this class. This may be surprising due to the fact that accelerated methods are readily available under the assumption of smoothness in $\smash{\ell_1}$-norm.
While Dragomir et al.~(Mathematical~Programming, 2022) have already showed that acceleration might be impossible under relative smoothness, their prox-function is pathological and constructed together with the hard instance. 
In contrast, we show non-acceleration for a {specific} prox-function with particularly favorable structure.
We also extend the result to the quantum setting, proving the same lower bound in the class of functions~$L$-smooth relative to the negative von Neumann entropy on the spectrahedron of~$d \times d$ Hermitian positive-semidefinite matrices with unit trace. 


\end{abstract}

%

\section{Introduction}
\label{sec:intro}

In this paper, we study convex optimization problems on the standard simplex~$\Simp \subset \R^d$, of the form
\begin{equation}
\label{eq:intro-problem}
\min_{x\in \Simp} f(x).
\end{equation}
More precisely, we are interested in the first-order oracle complexity\footnote{In the classical sense of Nemirovski and Yudin~\cite{nemirovski1983problem}; we shall briefly recap their complexity framework in Section~\ref{sec:prelim}.} of the natural problem class in which the objective smoothness is measured with respect to negative Shannon entropy
\begin{equation}
\label{eq:intro-entropy}
h: \Simp \to \R, \qquad h(x) = \sum_{k=1}^d (x)_k \log (x)_k,
\end{equation}
where~$(x)_k$ denotes the~$k{\textup{th}}$ entry of~$x$ and we use the standard convention~$0 \log 0 = 0$.
Here, ``smoothness of one function with respect to another'' refers to the notion of {\em relative smoothness}, as defined in~\cite{BauschkeBolteTeboulle2017} and~\cite{LuFreundNesterov2018},
which we now recall in the form adapted to our setting. 
Any function~$\phi$ strictly convex on a domain~$X \subseteq \R^d$ and $C^1$ in its relative interior, defines the {Bregman divergence}
\[
D_{\phi}: X \times \ri(X) \to \R_+, 
\qquad
D_{\phi}(x,u) =  \phi(x) - \phi(u) - \langle \bnabla\phi(u), x-u \rangle,
\]
a directed measure of discrepancy of an arbitrary point~$x \in X$ from the ``center'' point~$u \in \ri(X)$. 
(In the sequel,~$X$ is assumed to lie in some affine subspace of~$\R^d$, and we write~$\bnabla f(x)$ for the gradient of~$f: X \to \R$ restricted to the affine hull of~$X$; a more detailed discussion is deferred to Section~\ref{sec:prelim}.)
For~$L>0$, a function~$f: X \to \R$ in~$C^1(\ri(X))$ is called~{\em $L$-smooth relative to~$\phi$} if
\begin{equation}
\label{eq:relative-smoothness}
f(x) - f(u) - \inner{\bnabla f(u)}{x-u} \le LD_{\phi}(x,u) 
\qquad \forall (x,u) \in X \times \ri(X).
\end{equation}
The inequality in~\eqref{eq:relative-smoothness} can be recast as~$D_f(x,u) \le LD_{\phi}(x,u)$, so~\eqref{eq:relative-smoothness} is equivalent to the convexity of~$L\phi - f$.
When~$X = \R^d$ and~$\phi(\cdot) = \frac{1}{2} \|\cdot\|_2^2$, the associated Bregman divergence is~$\frac{1}{2}\|x-u\|_2^2$, and~$L$-relative smoothness reduces to the usual Euclidean smoothness, i.e.,~$L$-Lipschitzness of~$\bnabla f$.
When~$(X,\phi) = (\Simp,h)$, the corresponding Bregman divergence gives the Kullback--Leibler divergence. 
For our purposes, it is convenient to treat the latter as an extended-value function over~$\Simp \times \Simp$,
\begin{equation}
\label{eq:intro-kl}
        \Div(x,u)
        :=
        \begin{cases}
        \displaystyle
        \sum_{k \,\in\, \supp x} (x)_k\log\frac{(x)_k}{(u)_k},
        & \supp x\subseteq\supp u,\\[1.2ex]
        +\infty,
        & \text{otherwise},
        \end{cases}
\end{equation}
with the convention~$0\log 0 = 0$.
Geometrically,~$\Div(x,u) < +\infty$ occurs when~$u$ is ``at least as interior'' as~$x$;
on \(\Simp\times\ri(\Simp)\) this definition coincides with the usual Bregman divergence generated by~$h(\cdot)$. 

%
We are interested in the~$(\Simp,h)$ case, so let us formally define the corresponding class of functions.

\begin{definition}
\label{def:entropy-smoothness}
The class~$\Fh(L)$ consists of all functions~$f: \Simp \to \R$ that are convex, continuously differentiable on~$\ri(\Simp)$, and satisfy the following inequalities for all~$(x,u) \in \Simp \times \ri(\Simp)$:
\begin{equation} 
\label{eq:intro-sandwich} 
0 \le f(x)-f(u)-\inner{\bnabla f(u)}{x-u} \le L\Div(x,u).
\end{equation}
\end{definition}
\noindent
The first inequality in \eqref{eq:intro-sandwich} is the usual first-order convexity certificate, whereas the second one is the condition of~$L$-smoothness relative to~$h$ on~$\Simp$.
Note that differentiability is only required on \(\ri(\Simp)\); yet, functions in \(\Fh(L)\) are finite and continuous on the closed simplex. In particular,~$h\in\Fh(1)$.

The notion of relative smoothness, proposed simultaneously by~Bauschke et al.~\cite{BauschkeBolteTeboulle2017} and Lu et al.~\cite{LuFreundNesterov2018}, is motivated by the simple yet striking observation: inequality~\eqref{eq:relative-smoothness} 
can be interpreted as the proximal descent lemma formulated directly in terms of the Bregman divergence at hand.
Using this observation,~\cite{BauschkeBolteTeboulle2017} and~\cite{LuFreundNesterov2018} generalized the classical analysis of mirror descent~\cite{nemirovski1983problem}, showing that
\begin{equation}
\label{eq:intro-upper-general}
f(x_T)- \min_{x \in X} f(x)\le \frac{LD_{\phi}(x^\star,x_0)}{T} \qquad \forall x^\star \in \Argmin\limits_{x \in X} f(x)
\end{equation}
after~$T$ iterations of the algorithm initialized from~$x_0$,  all without using {\em any norm} on the domain.
Specializing this to the entropy on~$\Simp$ with initialization at the barycenter~$d^{-1} \ones$, we get the bound
\begin{equation}
\label{eq:intro-upper-KL}
f(x_T)- \min_{x \in \Simp} f(x)\le \frac{L\log d}{T} \qquad \forall f \in \Fh(L).
\end{equation}
To put this into perspective, the classic analysis of mirror descent 
proceeds by choosing a norm~$\|\cdot\|$ such that~$\phi$ is~$1$-strongly convex w.r.t.~$\|\cdot\|$ on~$X$, 
which implies convexity and~$L$-smoothness in~$\|\cdot\|$,
\begin{equation*}
0 \le f(x) - f(u) - \inner{\bnabla f(u)}{x-u} \le \frac{L}{2} \|x-u\|^2
\qquad \forall x,u \in X,
\end{equation*}
and using~$1$-strong convexity of~$\phi$ at some stage of the analysis. 
This route still results in~\eqref{eq:intro-upper-general}, but only in the smaller class of functions: indeed, the above inequality implies~\eqref{eq:relative-smoothness} but not vice versa.
In particular, in the relevant to us case~$(\Simp,h)$, the suitable norm is~$\|x\|_{1} = \sum_{i = 1}^d |(x)_i|$;
adopting the name~$\Vh(L)$ for the corresponding class of functions on~$\Simp$, i.e.~those satisfying the inequalities
\begin{equation}
\label{eq:intro-l1-smoothness}
0 \le f(x) - f(u) - \inner{\bnabla f(u)}{x-u} \le \frac{L}{2} \|x-u\|_1^2
\qquad \forall x,u \in \Simp,
\end{equation}
the inclusion~$\Vh(L) \subsetneq \Fh(L)$ follows from  the~$1$-strong convexity of~$h$ w.r.t.~$\ell_1$-norm, that is Pinsker's inequality~\cite{pinsker1964information}.  In these terms, the guarantee in~\eqref{eq:intro-upper-KL} extends from~$\Vh(L)$ to~$\Fh(L)$.

One may ask whether other classical results for proximal algorithms admit similar generalizations. 
In particular, a very natural question, mentioned already in~\cite{BauschkeBolteTeboulle2017},~\cite{LuFreundNesterov2018} and more explicitly discussed by R.-A.~Dragomir in his PhD thesis~\cite{DragomirThesis2021}, is whether acceleration ``\`a la Nesterov''~\cite{nesterov1983method} is possible under relative smoothness.  
Indeed, in the class of functions convex and~$L$-smooth on a set~$X$ with respect to a given norm~$\|\cdot\|$, 
one can obtain~$O(LD_{\phi}(x^\star,x_0)T^{-2})$ error using any potential that is~$1$-strongly convex with respect to the norm~$\|\cdot\|$; see~\cite{nesterov2013first}.\,\footnote{We write~$g = O(f)$ or~$f = \Omega(g)$ if there exists a universal constant~$c > 0$ such that~$g \le cf$ for all argument values.}
With this in mind, we are led to the question:
\begin{quote}
{\em
Given~a convex domain~$X \subseteq \R^d$ and a potential~$\phi$,  is there a first-order algorithm with guaranteed~$O(T^{-2})$ convergence in the corresponding class of relatively smooth functions?
\em}
\end{quote}
One may also specialize this question to {\em concrete} domain-potential pairs. 
In particular, for~$(\Simp,h)$
\begin{equation}
\label{eq:intro-upper-KL-acc}
f(x_T)- \min_{x \in \Simp} f(x) = O\left(\frac{L\log d}{T^2}\right) \qquad \forall f \in \Vh(L),
\end{equation}
and the question we are left with is the following one:
\begin{mdframed}
\begin{quote}
\begin{center}
{\em
Can the guarantee in~\eqref{eq:intro-upper-KL-acc} be extended to the class~$\Fh(L)$?
\em}
\end{center}
\end{quote}
\end{mdframed}
To the best of our knowledge, both these questions are open.
Our paper resolves the second question.

\paragraph{Our results.}
Focusing on the case of~$(\Simp,h)$, we answer the acceleration question in the negative.
A rigorous statement of our result relies on the basic notions of black-box complexity theory~\cite{nemirovski1983problem}, to be recapped in Section~\ref{sec:prelim}.
The simplified formulation, presented next, suffices to make our point.
\begin{theorem}
\label{th:intro-lower}
Let~$L > 0$,~$T \ge 1$, and~$d = \Omega(T^2)$.
For any deterministic method that makes~$T$ queries~$x_{1}, \dots, x_{T} \in \ri(\Simp)$ of the oracle~$x \mapsto (f(x),\bnabla f(x))$ and returns~$x_{T+1} \in \Simp$, there exists~$f(\cdot) \in \Fh(L)$ such that
\begin{equation}
\label{eq:intro-lower}
f(x_{T+1}) - \min_{x \in \Simp} f(x) > \frac{L}{4(T+1)}. 
\end{equation}
\end{theorem}
This result shows that, in contrast to~$\Vh(L)$, accelerated~$O(T^{-2})$ convergence cannot be attained on~$\Fh(L)$. 
Moreover, the convergence guarantee~\eqref{eq:intro-upper-KL} of entropic mirror descent with stepsize~$1/L$ is optimal---up to a logarithmic factor---if the dimension is large enough, namely when~$\smash{d =\Omega(T^2)}$.

Let us make several remarks regarding Theorem~\ref{th:intro-lower}. 
%
\begin{itemize}
\item[R1.] As mentioned, Theorem~\ref{th:intro-lower} leaves a logarithmic gap between the best available upper and lower bounds for the class~$\Fh(L)$; cf.~\eqref{eq:intro-upper-KL}. 
We expect that the lower bound in~\eqref{eq:intro-lower} is loose, and the missing~$\log d$ factor can be recovered.
A promising approach is outlined in Section~\ref{sec:outro}.

\item[R2.] When~$d = \exp(T)$, the right-hand side of~\eqref{eq:intro-lower} reads as~$O(\smash{{LT^{-2} \log d}})$, matching the upper bound in~\eqref{eq:intro-upper-KL-acc} and diverging from~\eqref{eq:intro-upper-KL} by~$T$.  Therefore, acceleration is formally not ruled out for such ``extremely high-dimensional'' problems; however, this regime is of limited practical interest anyway, since even a single iteration of a deterministic first-order method is prohibitive.

\item[R3.] The quadratic dependence of the dimension on~$T$  is crucial for our resisting oracle construction.
This gives an~$O(T^{-2})$ ``safety margin'' that ensures consistency of the transcript after~$T$ steps.

\item[R4.] Restriction to interior queries is natural: indeed,~$f$ might be non-differentiable on the relative boundary; this is not a pathological case either, as shown by the scaled negative entropy~$Lh$.
\vspace{-0.1cm}
\end{itemize}

Theorem~\ref{th:intro-lower} admits a noncommutative generalization, pertaining to functions  on the ``spectraplex''
\[
\Dens := \{X \in \Herm_+: \tr(X) = 1\}
\] 
where~$\Herm_+$ is the Hermitian positive-semidefinite cone. 
In this setting, defined rigorously in Section~\ref{sec:quantum}, 
we introduce the class~$\Qd(L)$ of functions on~$\Dens$ that are convex,~$C^1(\ri(\Dens))$, and~$L$-smooth relative to the negative von Neumann entropy~$\tr(X \log X)$, which is the spectral analog of~$h(x)$. 
It turns out that this noncommutative setting can be reduced to the diagonal case, corresponding to the class~$\Fh(L)$ and the setting of Theorem~\ref{th:lower}. 
This is done by embedding the hard instance of Theorem~\ref{th:intro-lower} diagonally, and using Lindblad's inequality~\cite{Lindblad1975} to argue that the composition~$F(\cdot) = f( \diag(\cdot))$ of~$f(\cdot) \in \Fh(L)$ and the diagonal extraction map belongs to~$\Qd(L)$. The final ingredient is the observation that the diagonal of the (matrix) transcript of an arbitrary first-order method run on~$\smash{F}$ emulates the transcript of some first-order method run on~$f$.
We defer further details to Section~\ref{sec:quantum}.

Finally, as a minor contribution, in Section~\ref{sec:outro} we find an explicit form of the pointwise minimal interpolant in the class~$\Fh(L)$ for given interpolation data~$\cI = (x_i,f_i,\xi_i)_{i = 1}^{N}$, defined as the function~$f \in \Fh(L)$ that interpolates~$\cI$ to first order, i.e.~satisfies~$f(x_i) = f_i$ and~$\nabla f(x_i) = \xi_i$ for all~$i \in [N]$, and is no larger than any other such function at every point~$x \in \Simp$. This result follows easily from the results in R.-A.~Dragomir's PhD thesis~\cite{DragomirThesis2021}; we record it due to its pivotal role in the promising approach of improving Theorem~\ref{th:intro-lower}. Further details are deferred to Section~\ref{sec:outro}.

\paragraph{Summary of the approach.}
Our construction is based on the {\em right Bregman--Moreau envelope}~\cite{BauschkeDaoLindstrom2018} applied to an incremental coordinatewise construction in the spirit of~Guzm\'an and Nemirovski~\cite{GuzmanNemirovski2015}.
The right Bregman-Moreau envelope replaces the inf-convolution smoothing operator of~\cite{GuzmanNemirovski2015} as the smoothing mechanism.
This smoothing mechanism heavily relies upon the $f$-divergence properties of the KL divergence, namely its joint and separate convexity in {both} arguments;
as demonstrated in~\cite{BauschkeBorwein2001}, among Bregman divergences, these properties occur only in
the Euclidean and KL cases.
Crucially, the right envelope of a convex function is convex and smooth with respect to the potential~(see~\cite{BauschkeDaoLindstrom2018});
meanwhile, the local smoothing of~\cite{GuzmanNemirovski2015} cannot be employed, as it would give an instance in~$\Vh(L)$, and the latter class {\em does} admit acceleration. 

\paragraph{Related work.}
%
%
In addition to~\cite{GuzmanNemirovski2015},  closely relevant to ours is the work of Dragomir, Taylor, d'Aspremont, and Bolte~\cite{DragomirTaylordAspremontBolte2022}, who showed the following: 
there {\em exists a pair}~$(f,\phi)$ in which~$\phi$ is strictly convex on the positive orthant~$\R^d_+$,~$f$ is convex and smooth relative to~$\phi$,
and any first-order method has worst-case convergence rate no better than~$\Omega(1/T)$. 
Their result is based on so-called performance estimation techniques (e.g.~\cite{DroriTeboulle2014},~\cite{TaylorHendrickxGlineur2017}), allowing one to find worst-case instances in infinite-dimensional functional classes by reducing the corresponding optimization problem to a (finite-dimensional) semidefinite program (SDP), dubbed ``performance estimation program'' (PEP).
Usually, the PEP methodology is limited to Euclidean geometry, since the SDP representation arises from the Gram matrix that juxtaposes the candidate iterates~$x_{k}$ and gradients~$g_{k} = \nabla f(x_{k})$ in the transcript; as a result, PEPs are poorly suited for dealing with non-Euclidean geometries, where the terms of the form~$\| x_k - x_l \|^2$ and~$\|g_k - g_l\|_*^2$ cannot be expressed linearly via dot products. 
The authors of~\cite{DragomirTaylordAspremontBolte2022} elegantly sidestepped this limitation by allowing the potential itself to vary with~$f$, so that the PEP constructs a worst-case {\em pair}~$(f,\phi)$. 
However, the resulting potential is quite pathological---as one might expect with its adversarial origin---and it may very well be that some {\em specific} pairs~$(X,\phi)$ do admit acceleration. 
While our result shows this is not the case for~$(\Simp, h)$, other highly structured settings remain open, most notably that of the logarithmic-barrier potential~$\phi(x) = -\sum_{i = 1}^d \log(x)_i$ on~$\R^d_+$.



Chapter~5 of Dragomir's PhD thesis~\cite{DragomirThesis2021} provides {exact}  first-order interpolation conditions for the class~$\smash{\Fh^+(L)}$ of functions on the nonnegative orthant~$\smash{\R^d_+}$ that are convex and smooth relative to the unnormalized entropy~$\smash{h_\e(x) := \sum_{k = 1}^d (x)_k \log (x)_k - (x)_k}$; such conditions are the inequalities imposed on the interpolation data~$\cI = \smash{(x_i,f_i,\xi_i)_{i = 1}^N}$, whose validity is equivalent to the existence of a function in~$\Fh^+(L)$ that interpolates~$\cI$. Conceptually, these conditions are counterparts of SDP-representable interpolation conditions in the Euclidean case, which are the crux of the PEP framework, and their existence crucially relies upon the joint convexity of KL divergence~(see \cite[Lem.~5.3.3 and Rem.~3]{DragomirThesis2021}). In Section~\ref{sec:outro}, we generalize these conditions for the class~$\Fh(L)$ of entropy-smooth functions on~$\Simp$, and then use them to derive the {pointwise minimal} interpolant.
 It appears that these results might be leveraged to 
recover the logarithmic factor missing in~\eqref{eq:intro-lower}. In Section~\ref{sec:outro} we further discuss this possibility in the light of the recent work of Florea and Nesterov~\cite{florea2025optimal}.

Finally, we mention the triangular scaling exponent framework of~\cite{HanzelyRichtarikXiao2021}, which seeks to relax the standard assumption of~$1$-strong convexity of~$\phi$ w.r.t.~a norm while retaining accelerated convergence. 
Our critique is that, while leading to locally adaptive algorithms that prove to be highly effective for optimization problems arising in some applications, the triangular scaling condition is hard to ensure in the worst case.
In particular, Theorem~\ref{th:intro-lower} shows that this condition is not satisfied for~$\Fh(L)$.

\paragraph{Roadmap.}
In Section~\ref{sec:prelim} we collect the ingredients for the proof of Theorem~\ref{th:intro-lower}.
To that end, we recall the properties of the right Bregman-Moreau envelope and the associated smoothing operator, formalize the oracle complexity framework, and describe the family of hard instances subsequently used in the proof.
In Section~\ref{sec:proof} we carry out the proof of Theorem~\ref{th:intro-lower}, and in Section~\ref{sec:quantum} we formulate and discuss its noncommutative generalization. 
In Section~\ref{sec:outro}, we derive exact interpolation conditions and the form of pointwise minimal interpolant in~$\Fh(L)$.

\paragraph{Notation.}
We denote~$[d] := \{1,2,\dots, d\}$. We let~$e_i$ be the~$i$th canonical basis vector in~$\R^d$ and use the concise notation~$(x)_i := \langle x, e_i \rangle$ for the~$i$th entry of~$x \in \R^d$. We let~$\ones$ be the all-ones vector in~$\R^d$. As previously mentioned, we write~$\bnabla f$ for the {tangent} gradient of~$f \in C^1(\ri(\Simp))$; in particular,~$\bnabla f(x) \in \spanof{\ones}^\perp$ for any~$x \in \ri(\Simp)$. 
Additional notation is introduced as necessary.


\section{Building blocks}
\label{sec:prelim}

\paragraph{Tangent gradients.}
In Sections~\ref{sec:prelim}--\ref{sec:proof}, we work with functions defined on~$\Simp$ and differentiable in~$\ri(\Simp)$; this includes every~$f \in \Fh(L)$ and, in particular, the hard instances presented in Section~\ref{sec:hard-subclass}. 
While such functions may arise as restrictions of~$C^1(\R_{++}^d)$-functions (as, for example,~$h$ itself), this is not required in general. For such a function~$f$, we might differentiate it at~$x \in \ri(\Simp)$ intrinsically over the affine hull of~$\Simp$, by identifying the affine subspace~$\spanof{\ones}^{\perp} + d^{-1} \ones$  with~$\R^{d-1}$, via an affine homeomorphism (with orthogonal linear transformation), taking the full gradient of the resulting function in~$\R^{d-1}$, and changing the basis to account for the transformation. 
We take this as the definition of the tangent gradient~$\bnabla f$. 
Note that~$\bnabla f$ always belongs to the tangent space~$\Tan := \spanof{\ones}^\perp$; moreover, if~$f$ is defined in a proper~$\R^d$-neighborhood of~$x$, then~$\bnabla f(x)$ coincides with the Euclidean projection onto~$\Tan$ of the full gradient; denoting the latter with~$\nabla_\star f$, 
\begin{equation}
\label{eq:grad-tangent}
\bnabla f(x) := \nabla_\star f(x) - d^{-1}  \langle \nabla_\star f(x), \ones\rangle \ones, \qquad x \in \ri(\Simp).
\end{equation}
In fact, even if~$f$ is defined only over~$\Simp$, this formula remains valid if we extend~$f$ to~$\R^d_+$ with differentiability over~$\R^d_{++}$. Such an extension is not unique, and~$\nabla_* f$ in general depends on the chosen extension; however,  the dependence is only in the normal component~$d^{-1}\langle \nabla_\star f(x), \ones \rangle \ones$, whereas~$\bnabla f(x)$ is an invariant depending only on the values of~$f$ in a~$\Tan$-neighborhood of~$x \in \ri(\Simp)$.


\subsection{Smoothing with right Bregman--Moreau envelope}
\label{sec:envelope}

Let $g: \Simp \to \mathbb{R}$ be convex and continuous.
Adapting the standard definition to our setting, {\em the right Bregman--Moreau envelope} of~$g$ with parameter~$\eta > 0$ is defined by its values on the simplex:
\begin{equation}
\label{eq:envelope}
\Env_\eta[g](x) := \min_{u \in \Simp}
\left\{g(u)+\eta^{-1} \Dh(x,u)\right\},
\qquad x \in \Simp.
\end{equation}
For~$x \in \ri(\Simp)$, the minimizer is unique and attained on~$\ri(\Simp)$; 
this defines~$\prox_{\eta g}: \ri(\Simp) \to \ri(\Simp)$,
\begin{equation}
\label{eq:prox}
\prox_{\eta g}(x) = \argmin_{u \in \Simp} \left\{ \eta g(u) + \Div(x,u) \right\}, \qquad x \in \ri(\Simp),
\end{equation}
called {\em the prox-mapping} of $\eta g(\cdot)$.
%

Next, we collect some properties of the right Bregman--Moreau envelope and proximal mapping.

\begin{lemma}
\label{lem:prox-mapping}
For~$\eta > 0$, the prox-mapping~$\prox_{\eta g}(\cdot)$ is well-defined and continuous on~$\ri(\Simp)$.
\end{lemma}

\begin{proof}
As a continuous function on~$\Simp$,~$g$ is bounded from below.
From this and the expression~\eqref{eq:intro-kl} for KL divergence, we see that for~$x \in \ri(\Simp)$, minimization in~\eqref{eq:envelope} can be restricted to~$\ri(\Simp)$.
The same expression shows that the Hessian of~$\Div(x,\cdot)$ is diagonal with positive entries over~$\ri(\Simp)$, so the objective in~\eqref{eq:envelope} is strictly convex and the minimizer is unique.\footnote{Note that~$\Div(x,\cdot)$ is infinitely differentiable as a function over~$\R^d_{++}$, hence we may consider its full Hessian here.}
Continuity of~$\prox_{\eta g}(\cdot)$ follows from the optimal-set mapping theorem of Rockafellar and Wets (see~\cite[Example~5.22]{rockafellar1998variational}), applied to
\[
        \Phi(u,x)
        :=
        g(u)+\eta^{-1}\Div(x,u)+\delta_{\Simp}(u),
        \qquad x\in\ri(\Simp).
\]
Since~$\Simp$ is compact, the function $\Phi$ is proper, lower semicontinuous, and level-bounded in $u$ locally uniformly in $x$.
Since $\Phi(\prox_{\eta g}(\bar x),x)$ is continuous in $x$ at $\bar x$, the cited theorem implies continuity of the value function at any~$\bar x\in\ri(\Simp)$, and outer continuity of the corresponding minimizing set-mapping.
Since this mapping outputs a singleton, this is the continuity of~$x \mapsto \prox_{\eta g}(x)$.
\end{proof}

\begin{lemma}
\label{lem:envelope-regularity}
The function~$\Env_\eta[g](\cdot)$ is convex on~$\Simp$ and continuously differentiable on~$\ri(\Simp)$, with
\begin{equation}
\label{eq:right-bm-prox-identity}
    \bnabla \Env_\eta[g](x)
    =
    -\eta^{-1}(\log\prox_{\eta g}(x) - \log x),
\end{equation}
where~$\log(\cdot)$ is applied entrywise. 
Moreover,~$\Env_\eta[g](\cdot)$ satisfies the inequality for $x \in \ri(\Simp)$ and~$y \in \Simp$:
\begin{equation}
\label{eq:right-bm-relative-smoothness}
    \Env_\eta[g](y)
    \le
    \Env_\eta[g](x)
    +\left\langle \nabla \Env_\eta[g](x),y-x\right\rangle
    +\eta^{-1}\Div(y,x).
\end{equation}
In particular,~$\Env_\eta[g](\cdot) \in \Ent(\eta^{-1})$. 
\end{lemma}

\begin{proof}
By the perspective rule applied coordinatewise to~$x \mapsto x \log x$, extended-value KL divergence~\eqref{eq:intro-kl} is jointly convex, whence convexity of $\Env_\eta[g]$ follows by the partial minimization rule. 
For~\eqref{eq:right-bm-prox-identity}, apply Danskin's theorem and the first-order optimality condition to~\eqref{eq:envelope} with~$x \in \ri(\Simp)$.
Finally, invoking~\eqref{eq:envelope} at~$x$ and~$y$, with~$u := \prox_{\eta g}(x)$ as a feasible point in the latter case, gives
\begin{equation}
\begin{aligned}
\label{eq:right-bm-one-step}
    \eta \Env_\eta[g](y)- \eta \Env_\eta[g](x)
    \,\le\,
    \Div(y,u)-\Div(x,u)
    \,=\,
    \Div(y,x)+ \left\langle \log x-\log u,y-x\right\rangle.
\end{aligned}
\end{equation}
We used the three-point identity in the final step.
Combining this with~\eqref{eq:right-bm-prox-identity} and~\eqref{eq:grad-tangent} gives~\eqref{eq:right-bm-relative-smoothness}.
\end{proof}

Next, we establish {\em locality} of the smoothing mechanism based upon the right Bregman--Moreau envelope, ensuring that the envelopes of locally coinciding functions (locally) agree to first order.

\begin{lemma}
\label{lem:right-bm-locality}
Let~$g,\tilde g:\Simp\to\mathbb{R}$ be finite, continuous, convex, and such that~$\tilde g\ge g$ on~$\Simp$. 
Fix arbitrary~$\hat x\in\ri(\Simp)$ and let~$\hat u:=\prox_{\eta g}(\hat x)$. 
If $\tilde g=g$ in a relative neighborhood of $\hat u$, then 
\begin{equation}
\label{eq:right-bm-locality-equality}
    \Env_\eta[\tilde g](x)=\Env_\eta[g](x)
    \qquad \forall x \in \mathcal{X}
\end{equation}
in some relative neighborhood $\mathcal{X}$ of $\hat x$. 
In particular,~$\bnabla\Env_\eta[\tilde g](\hat x)= \bnabla \Env_\eta[g](\hat x)$.
\end{lemma}

\begin{proof}
Let~$\mathcal{U}$ be a relative neighborhood of~$\hat u$ where~$\tilde g=g$. 
By continuity of~$\prox_{\eta g}(\cdot)$, see Lemma~\ref{lem:prox-mapping}, in some relative neighborhood $\mathcal{X}$ of $\hat x$ one has
\[
\prox_{\eta g}(x) \in \mathcal{U} \qquad \forall x \in \mathcal{X}.
\]  
Now, fix arbitrary~$x \in \mathcal{X}$ and put~$u := \prox_{\eta g}(x)$. 
Since $\tilde g \ge g$ everywhere on~$\Simp$, by~\eqref{eq:envelope} we get
\begin{equation*}
    \Env_\eta[\tilde g](x)\ge \Env_\eta[g](x).
\end{equation*}
On the other hand, by invoking~\eqref{eq:envelope} for~$\tilde g$ with~$u$ as a feasible point (note that~$u \in \mathcal{U}$), we get
\begin{equation*}
\begin{aligned}
    \Env_\eta[\tilde g](x)
    &\le
    \tilde g(u)+\eta^{-1}\Div(x,u)
    \\
    &=
    g(u)+\eta^{-1}\Div(x,u)
    =
    \Env_\eta[g](x).
\end{aligned}
\end{equation*}
Thus~$\Env_{\eta} [\tilde g] = \Env_{\eta} [g]$ on~$\mathcal{X}$, and~$\bnabla\Env_\eta[\tilde g](\hat x)= \bnabla \Env_\eta[g](\hat x)$ follows from differentiability on $\ri(\Simp)$.
\end{proof}

\subsection{Oracle complexity model and formal statement of the result}
\label{sec:formal}

We work in the oracle complexity model of~\cite{nemirovski1983problem}, adapted to account for the lack of differentiability on the boundary for objectives in~$\Fh(L)$. 
Let us give a brief yet rigorous summary of this model.

Any~$f \in \Fh(L)$ specifies a {\em first-order oracle}~$\Orc_f$ which, when queried at any~$x \in \ri(\Simp)$, returns 
\begin{equation}\label{eq:oracle}
\Orc_f(x) := (f(x), \bnabla f(x))
\end{equation}
where~$\bnabla f$ is the tangent gradient of~$f$. 
As we explained in the beginning of Section~\ref{sec:prelim}, the use of~$\bnabla f$ is without loss of generality, since~$f$ is only defined over~$\Simp$ whose affine hull is parallel to~$\Tan$. 
However, one could still ask if anything could be gained by allowing the objective to be defined on the whole orthant~$\R^d_{+}$, with access to the full gradient oracle, while only requiring the relative smoothness condition in~\eqref{eq:intro-sandwich} to hold on~$\Simp$.
To address this, in Appendix~\ref{app:conic} we show that~\eqref{eq:intro-lower} remains valid for minimization in the class of~$1$-homogeneous extensions~\cite{ovcharov2018proper} of functions in~$ \Fh(L)$.



A deterministic {\em $T$-step first-order method} (FOM) makes~$T$ sequential queries~$x_1, \dots, x_{T} \in \ri(\Simp)$ of such an~$\Orc_f$, and outputs~$x_{T+1} \in \Simp$. 
Thus, a~$T$-step FOM is specified by a collection of mappings
\begin{equation}
\label{eq:method-maps}
\begin{aligned}
\Phi_t: \Omega^{t-1}  &\to \ri(\Simp), \qquad t \in [T]; \\
\bar \Phi_{T+1}: \Omega^T &\to \Simp,
\end{aligned}
\end{equation}
where~$\Omega = \R \times \Tan$, and~$\Omega^0$ is a singleton (so~$\Phi_1$ merely selects a specific point~$x_1 \in \ri(\Simp)$).
We let~$\FOM_d(T)$ be the class of all~$T$-step FOMs, as per~\eqref{eq:method-maps}. 
When a given method~$\Mtd \in \FOM_d(T)$ is instantiated with an oracle~$\Orc_f$ associated to a specific instance~$f \in \Ent(L)$, the resulting sequence
\begin{equation}
\label{eq:method-iterates}
x_1^{\Mtd} = \Phi_1^{\vphantom\Mtd}, 
\;\;
x_2^{\Mtd}(f) = \Phi_2^{\vphantom\Mtd}(\Orc_f(x_1^{\Mtd})), 
\;\; \ldots, \;\; 
x_{T+1}^{\Mtd}(f) = \bar\Phi_{T+1}^{\vphantom\Mtd}(\,\Orc_f(x_1^\Mtd),\; \dots,\; \Orc_f(x_{T}^\Mtd(f))),
\end{equation}
along with the responses~$\Orc_f(x_1^\Mtd),\Orc_f(x_2^\Mtd(f)),\dots,\Orc_f(x_{T}^\Mtd(f))$, is called {\em the transcript} of~$\Mtd$ on~$f$.
Note that~$\Mtd$ selects each~$x_t$ {before} receiving the oracle response~$\Orc_f(x_t)$, and the crucial property of a transcript is its {\em  consistency}: the oracle responses in~\eqref{eq:method-iterates} correspond to {the same}~$f \in \Fh(L)$.

To quantify the hardness of first-order optimization over~$\Fh(L)$, we define its {\em minimax~$T$-risk}
\begin{equation}
\label{eq:minimax-risk}
\Risk_{d}(T,L) \; 
:=
\inf_{\Mtd^{\vphantom{2^L}} \,\in\, \FOM_{d}(T)} 
\quad
\sup_{f \,\in\, \Fh(L)} 
\quad
\left\{ f(x_{T+1}^\Mtd(f))- \min_{x \in \Simp} f(x) \right\}.
\end{equation}
With this definition at hand, we are now in the position to give the rigorous statement of Theorem~\ref{th:intro-lower}.

\begin{theorem}
\label{th:lower}
For all~$L > 0$,~$T \ge 1$ and~$d \ge 8(T+1)^2 + T$, the minimax~$T$-risk of~$\Fh(L)$ satisfies
\begin{equation}
\label{eq:lower}
\Risk_{d}(T,L) > \frac{L}{4(T+1)}. 
\end{equation}
\end{theorem}
To prove Theorem~\ref{th:lower}, we shall proceed via the {``resisting oracle''} approach: 
interacting with arbitrary~$\Mtd \in \FOM_d(T)$, we shall construct a sequence~$\omega_1, \dots, \omega_T \in \Omega$, with
$
\omega_{t} = \omega_t(x_1, \dots, x_t),
$
which, along with the associated sequence of queries
$
\smash{
x_1 = \Phi_1,\;
x_2 = \Phi_2(\omega_1),
\ldots,
x_{T+1} = \bar\Phi_{T+1}(\omega_1,\dots,\omega_{T}), 
}
$
corresponds to a consistent transcript, i.e.,~matches some~$f \in \Fh(L)$ in the sense that~$\omega_t = \Orc_f(x_t^\Mtd)$. 
Our hard instance will be a maximum of affine functions, smoothed via the right Bregman-Moreau envelope to put it in~$\Fh(L)$.
In what follows, we first describe the general family of such nonsmooth functions and study their envelopes, then present the adaptive construction and complete the proof.

Before we proceed, a simple remark is in order. The minimax~$T$-risk is~$1$-homogeneous in~$L$, i.e.
\begin{equation}
\label{eq:risk-homogeneity}
\Risk_d(T,L) = L\Risk_d(T,1),
\end{equation}
as seen by noting that~$f \in \Fh(1)$ implies~$Lf \in \Fh(L)$, and that passing from~$f$ to~$Lf$ does not influence consistency of a transcript.
Since the right-hand side of~\eqref{eq:lower} has the same homogeneity, it suffices to treat the~$L = 1$ case; in other words, exhibit~$f \in \Ent(1)$ that certifies~\eqref{eq:lower} with~$L = 1$.


\subsection{Hard instances and their properties}
\label{sec:hard-subclass}

Our hard instances are based on nonsmooth functions that are maxima of shifted coordinate atoms
%
\begin{equation}
\label{eq:backbone}
        g(u)=\max_{k \in [K]}\bigl\{-(u)_{j_k} - b_k\bigr\},
\end{equation}
indexed by a finite set~$[K]$, with distinct coordinates $j_k$ and offsets~$b_k \ge 0$. 
For brevity, we shall refer to such functions as {\em backbones}.
In a backbone, each atom is~$1$-Lipschitz in~$\ell_\infty$-norm, and so is the whole backbone; this will be used later on.
By definition, the set of {\em $g$-active atoms} at~$u \in \ri(\Simp)$ is
\[
A_g(u):=\bigl\{k\in [K]:\ -(u)_{j_k}-b_k =  g(u)\bigr\},
\]
and~$\smash{\{j_k: k \in A_g(u)\}}$ are {\em $g$-active coordinates~at~$u$}.
To obtain a legitimate~$\smash{f \in \Fh(1)}$, we take the right Bregman--Moreau envelope of a backbone~$g$.
Indeed, by Lemma~\ref{lem:envelope-regularity} we have~$\smash{\eta S_\eta[g](\cdot) \in \Fh(1)}$ for any~$\eta > 0$; later on, we shall fix~$\eta = {1}/{2}$.

%
%
Next, we prove some regularity properties of the right Bregman-Moreau envelope of a backbone.
We begin with a lemma that controls the entrywise growth of the prox-mapping associated with~$\eta g$.

\begin{lemma}
\label{lem:entry-growth}
Let $g$ be as in~\eqref{eq:backbone}, and~$\eta < 1$. 
For all~$x\in\ri(\Simp)$ and~$j \in [d]$, 
\begin{equation}
\label{eq:entry-growth}
0 < \eta (\prox_{\eta g}(x))_j 
\le \frac{\eta}{1-\eta} (x)_j.
\end{equation}
In particular, for~$\eta \le \frac{1}{2}$ one has~$0 < \eta \prox_{\eta g}(x) \le x$ entrywise.
\end{lemma}


\begin{proof}
Let~$\hat u := \prox_{\eta g}(x)$.
The left inequality in~\eqref{eq:entry-growth}, i.e.~that~$\prox_{\eta g}(x) \in \ri(\Simp)$, is immediate: 
if~$\hat u$ has a zero entry, then~$\Div(x,\hat u) = +\infty$ by~\eqref{eq:intro-kl}, which contradicts the optimality of~$\hat u$. 
For the right inequality, since the minimum in~\eqref{eq:envelope} is attained on~$\smash{\ri(\Simp)}$, the optimality condition writes as
\begin{equation}
\label{eq:optimality-for-g}
\exists \mu \in \R: \quad \partial g(\hat u) \ni \eta^{-1} \frac{x}{\hat u} - \mu \ones,
\end{equation}
where the division is entrywise.
Now, let~$(j_k)_{k \in A_g(\hat u)}$ be the active part of the finite sequence~$(j_k)_{k \in [K]}$ in~\eqref{eq:backbone}, 
and consider arbitrary subgradient
$\xi(\hat u) = -\sum_{k \in A_g(\hat u)} w_k e_{j_k}$ of~$g(\cdot)$ at~$\hat u$, where~$w_k \ge 0$ and~$\smash{\sum_{k \in A_g(\hat u)} w_k = 1}$. 
Letting~$\lambda_i$, for~$i \in [d]$, be the sum of all active multipliers for each coordinate~$i$,
\[
\lambda_i := \sum_{k \,\in\, A_g(\hat u)} w_k\, \delta_{i j_k},
\]
we can write~$\xi(\hat u) = -\sum_{i \in [d]} \lambda_i e_i$ where~$(\lambda_1, \dots, \lambda_d) \in \Simp$ (in particular,~$\lambda_i \le 1$).
Plugging in~\eqref{eq:optimality-for-g} the relevant subgradient~$\xi(\hat u)$---i.e., one realizing the inclusion in~\eqref{eq:optimality-for-g}---and rearranging, we get
\[
\frac{(x)_i}{(\hat u)_i} = \eta (\mu -\lambda_i), \quad i \in [d].
\]
Multiplying by~$(\hat u)_i$, summing over~$i$, and using that~$x,\hat u \in \Simp$, we see that~$\eta(\mu - \hat m_{\lambda}) = 1$, where~$\hat m_{\lambda} := \sum_{i \in [d]} \lambda_i (\hat u)_i$ 
is the $\lambda$-weighted average of the entries of~$\hat u$.
Combining with the above, 
\begin{equation}
\label{eq:ratio}
\frac{(x)_i}{(\hat u)_i} = 1 + \eta(\hat m_{\lambda} - \lambda_i), \quad i \in [d].
\end{equation}
Noting that~$\hat m_{\lambda} \ge 0$ and~$\lambda_i \le 1$, we get~$(\hat u)_i \le (1-\eta)^{-1} (x)_i$, that is the right inequality in~\eqref{eq:entry-growth}.
\end{proof}

The next lemma controls the value decrease of a backbone caused by smoothing.

\begin{lemma}
\label{lem:loss}
Let~$g$ be as in~\eqref{eq:backbone},~$x\in\ri(\Simp)$, 
and~$\hat u =\prox_{\eta g}(x)$. 
Furthermore, suppose that~$(\hat u)_{j_k} \le M$ for all~$g$-active coordinates at~$\hat u$, i.e.~for all~$\smash{k \in A_g(\hat u)}$.
Then
\begin{equation}
\label{eq:loss-bound}
        \Env_\eta[g](x)\ge g(x)-\eta M.
\end{equation}
\end{lemma}

\begin{proof}
Rearranging \eqref{eq:ratio} gives $(x)_i-(\hat u)_i=\eta (\hat m_\lambda-\lambda_i) (\hat u)_i$ in terms of~$\hat m_{\lambda} = \sum_{i \in [d]} \lambda_i (\hat u)_i$, the quantity introduced in the proof of Lemma~\ref{lem:entry-growth}. 
Whence
\begin{equation}
\label{eq:displacement}
        \abs{(x)_i-(\hat u)_i}=\eta \abs{\hat m_\lambda-\lambda_i} (\hat u)_i \, .
\end{equation}
Since~$\hat m_\lambda$ is a convex combination of the active-coordinate values~$\{(\hat u)_{j_k}, k \in A_g(\hat u)\}$, we have~$\hat m_\lambda\le M$ by the premise of the lemma.
Let us show that
\begin{equation}
\label{eq:uniform-bound}
\| x - \hat u \|_{\infty} \le \eta M.
\end{equation}
To that end, since~$\hat m_\lambda, \lambda_i \in [0,1]$,
we get~$\abs{(x)_i - (\hat u)_i} \le \eta M$ for active coordinates~$i \in \{j_k, k \in A_g(\hat u)\}$. 
Meanwhile, for inactive coordinates~$\lambda_i = 0$, so~\eqref{eq:displacement} and~$(\hat u)_i \le 1$ result in
$
\abs{(x)_i - (\hat u)_i} \le \eta \hat m_{\lambda} \le \eta M.
$
This verifies~\eqref{eq:uniform-bound}. 
Now, by the $1$-Lipschitzness of a backbone, it follows that
$
g(\hat u)\ge g(x)-\eta M.
$
In turn, this implies~$\Env_\eta[g](x)=g(\hat u)+\eta^{-1}\Div(x,\hat u) \ge g(\hat u) \ge g(x)-\eta M$, as claimed in~\eqref{eq:loss-bound}.
\end{proof}

\section{Proof of Theorem~\ref{th:intro-lower}}
\label{sec:proof}
In this section, we prove Theorem~\ref{th:lower}, and Theorem~\ref{th:intro-lower} along with it.
To that end, we first implement the resisting oracle announced in Section~\ref{sec:formal}: interacting with an arbitrary method~$\Mtd\in\FOM_d(T)$, cf.~\eqref{eq:method-maps}--\eqref{eq:method-iterates}, after each query we add a new atom at the currently smallest-mass coordinate, with offset increased in constant increments. 
Smoothing via the right Bregman--Moreau envelope with~$\eta = \half$, while producing an instance in~$\Fh(1)$ by Lemma~\ref{lem:envelope-regularity}, ensures that a new atom is strictly dominated near all the previous queries; this effect is attained through~Lemma~\ref{lem:loss} and offsets. 
This guarantees transcript consistency: the final instance matches the earlier oracle answers.


The argument proceeds in three stages. 
In Section~\ref{sec:adaptive-backbone}, we detail the construction. 
In Section~\ref{sec:transcript-consistency}, we prove a locality lemma ensuring that new atoms are dominated by the previous ones, and deduce transcript consistency. 
In Section~\ref{sec:gap-proof}, we bound the suboptimality gap; this is done by augmenting the exposed face with an extra dimension and using the center of the augmented face.


\subsection{Constructing the hard instance}
\label{sec:adaptive-backbone}

Recall the assumption~$d \ge 8(T+1)^2+T$ in the premise. Throughout, fix~$\Mtd\in\FOM_d(T)$ and set
\begin{equation}
\label{eq:params}
        \veps_T :=\frac{1}{8(T+1)^2}\,,
        \qquad
        \rho:=\frac{2}{3}\,,
        \qquad
        \delta_T :=(2+\rho)\veps_T =\frac{1}{3(T+1)^2}\,.
        \qquad
\end{equation}
We shall select distinct coordinates~$j_1,\dots,j_{T},j_{T+1}$ incrementally, in response to~$\Mtd$'s queries. 

{\em Rounds~$t\in [T]$.} Once~$\Mtd$ has issued a query~$x_t\in\ri(\Simp)$ determined via~\eqref{eq:method-iterates} by the previous answers, 
we construct the answer to~$x_t$ by selecting a coordinate where~$x_t$ has the smallest mass:
\begin{equation}
\label{eq:fresh}
        j_t \in \Argmin_{i \,\in\, [d] \setminus E_{t}} \; \langle x_t, e_i \rangle, \qquad E_{t} := \{j_1,\ldots,j_{t-1}\},
\end{equation}
%
where~$E_{1} = \emptyset$ by convention.
We then let
\begin{equation}
\label{eq:atoms}
        a_t(x):=-\langle x, e_{j_t} \rangle -(t-1)\delta_T, \qquad 
        g_t(x) :=\max_{s \in [t]} a_s(x),
        \qquad
        f_t := \tfrac{1}{2}\Env_{1/2}[g_t]
\end{equation}
and return
\begin{equation}
\label{eq:oracle-answers}
\Orc_{f_t}(x_t) = (f_t(x_t), \nabla f_t(x_t)),
\end{equation}
cf.~\eqref{eq:oracle}, as the answer to query~$x_t$. 
Note that~$g_t$, cf.~\eqref{eq:atoms}, is a backbone in the sense of~\eqref{eq:backbone}. 
This procedure specifies the resisting oracle at rounds~$t \in [T]$, corresponding to~$\Mtd$'s queries~$x_1, \dots, x_T$.

{\em Final round.}
To finalize the construction and exhibit the hard instance, we use the candidate minimizer~$x_{T+1} \in \Simp$ returned by~$\Mtd$ after the final oracle call.
Namely, we invoke~\eqref{eq:fresh} with~$t = T+1$, producing
\[
j_{T+1} \in \Argmin_{i \,\in\, [d] \setminus E_{T}} \; \langle x_{T+1}, e_i \rangle;
\]
we then let~$g_{T+1}(x) := \max\{g_T(x), a_{T+1}(x)\}$ with~$a_{T+1}(x):=-\langle x, e_{j_{T+1}} \rangle - T\delta_T$, cf.~\eqref{eq:atoms}, and take
\begin{equation}
\label{eq:final-instance}
f_{T+1} := \tfrac{1}{2} \Env_{1/2}[g_{T+1}]
\end{equation}
as the hard instance. Notice that~$g_{T+1}$ is still a backbone, therefore~$f_{T+1} \in \Fh(1)$ by Lemma~\ref{lem:envelope-regularity}.

Before we begin the proof of transcript consistency, let us record one implication of Lemma~\ref{lem:entry-growth}.

\begin{proposition}
\label{prop:fresh-entries}
For every~$t \in [T]$, selection rule~\eqref{eq:fresh} guarantees the following for~$\veps_T$ as in~\eqref{eq:params}:
\begin{equation}
\label{eq:fresh-entries}
\langle \prox_{\half g_{t}}(x_t), e_{j_t} \rangle 
\le 2\langle x_t, e_{j_t} \rangle 
\le 2\veps_T.
\end{equation}
Moreover, the right-hand inequality in~\eqref{eq:fresh-entries} extends to~$t = T+1$, i.e.~it holds that~$\langle x_{T+1}, e_{j_{T+1}} \rangle \le \veps_T.$ 
\end{proposition}
\begin{proof}
For any~$t \in [T+1]$, the cardinality of exposed support~$E_t = \{j_1,\dots,j_{t-1}\}$ is at most~$T$, so that of~$H_t := [d] \setminus E_t$ is at least
$
d-T \ge \veps_T{}^{-1},
$
cf.~\eqref{eq:params}. 
As such,~$\min_{i \in H_t} \langle x_t, e_i \rangle > \veps_T$ would imply
\[
\langle x_t, \ones \rangle \ge \sum_{i \in H_t} \langle x_t, e_i \rangle \ge |H_t| \, \min_{i \in H_t} \, \langle x_t, e_i \rangle > 1,
\]
contradicting~$x_t \in \Simp$. This proves the right-hand inequality in~\eqref{eq:fresh-entries} for all~$t \in [T+1]$. 
For~$t \in [T]$, we additionally have~$x_{t} \in \ri(\Simp)$, and the left-hand inequality in~\eqref{eq:fresh-entries} follows by Lemma~\ref{lem:entry-growth}.
\end{proof}


\subsection{Ensuring transcript consistency}
\label{sec:transcript-consistency}

In this section, we establish transcript consistency for the construction presented in Section~\ref{sec:adaptive-backbone}.
The crux of the argument is that the ``fresh'' atom introduced at round~$t$ remains strictly dominated (``locally invisible'') in a vicinity~$\mathcal{U}_s$ of~$\smash{\hat u_s = \prox_{\half g_s}(x_s)}$; thus,~$g_t = g_s$ over~$\mathcal{U}_s$. 
By Lemma~\ref{lem:right-bm-locality}, this gives local coincidence of~$f_s = S_{1/2}[g_s]$ and~$f_{T+1} = S_{1/2}[g_{T+1}]$ in some relative neighborhood of~$x_s$.

\begin{lemma}
\label{lem:local-invisibility}
For all~$1\le s<t\le T+1$, the following holds.
\vspace{-0.1cm}
\begin{enumerate}
\item 
The point~$\hat u_s = \prox_{\half g_s}(x_s)$ satisfies\vspace{-0.2cm}
\begin{equation}
\label{eq:margin}
         a_t(\hat u_s) \le g_s(\hat u_s) - \rho\veps_T. \vspace{-0.1cm}
\end{equation}
\item As a result, there is a relative neighborhood~$\mathcal{X}_s \subseteq\ri(\Simp)$ of~$x_s$ where one has~$f_t(x) = f_s(x)$.
\end{enumerate}
\end{lemma}

\begin{proof}
\proofstep{1}. 
By~\eqref{eq:atoms}, a backbone dominates its last atom, so
$
g_s(\hat u_s) \ge a_s(\hat u_s) = -\langle \hat u_{s}, e_{j_s} \rangle - (s-1)\delta_T.
$
When combined with the bound~$\langle \hat u_{s}, e_{j_s} \rangle \le 2\veps_T$ furnished by Proposition~\ref{prop:fresh-entries}, cf.~\eqref{eq:fresh-entries}, this gives
\[
g_s(\hat u_s)\ge-2\veps_T-(s-1)\delta_T.
\] 
On the other hand, since~$\hat u_{s}$ has nonnegative entries,~$a_t(\hat u_s)=-\langle \hat u_{s}, e_{j_t}\rangle - (t-1)\delta_T \le-(t-1)\delta_T$. 
Subtracting the two estimates and using the facts that~$t > s$ and~$\delta_T=(2+\rho)\veps_T$, cf.~\eqref{eq:params}, we get
\begin{equation*}
        g_s(\hat u_s)-a_t(\hat u_s)
        \;\ge\;
        (t-s)\delta_T-2\veps_T
        \;\ge\;
        \delta_T-2\veps_T
        \;=\; \rho\veps_T.
\end{equation*}

\noindent\proofstep{2}. 
Our plan is to invoke Lemma~\ref{lem:right-bm-locality}.
Fix~$s<t$ and consider the neighborhood~$\mathcal{U}_s$ of~$\hat u_s$ as follows:
\[
\mathcal{U}_s 
:= 
\left\{
	u \in \ri(\Simp): \|u - \hat u_s\|_{\infty} \le \frac{\rho\veps_T}{4}
\right\}.
\]
Combining~\eqref{eq:margin} with~$1$-Lipschitzness of~$a_t(\cdot)$ and~$g_s(\cdot)$ in~$\ell_\infty$-norm, one has for all~$u \in \mathcal{U}_s$ and~$t > s$:
\[
        g_s(u)-a_t(u)
        \ge
        g_s(\hat u_s)-a_t(\hat u_s)-\frac{\rho\veps_T}{2}
        \stackrel{\eqref{eq:margin}}{\ge}
        \rho\veps_T - \frac{\rho\veps_T}{2}
        =\frac{\rho\veps_T}{2}
        >0.
\]
Thus, all future atoms~$a_t$,~$t > s$, are strictly dominated by~$g_s$ in the relative neighborhood~$\mathcal{U}_s$ of~$\hat u_s$ (note that we used that~$\hat u_s$ is in the relative interior of~$\Simp$). 
Applying this to~$t = s+1, \ldots, T+1$, 
\begin{equation*}
        g_{T+1}(u)
        =\max \left\{g_s(u), \max_{s < t \le T+1} a_t(u)\right\}
        =g_s(u) 
      \qquad \forall u \in \mathcal{U}_s.
\end{equation*}
Meanwhile, we have~$g_{T+1} \ge g_s$ pointwise on~$\Simp$, directly from the definitions of~$g_{T+1}$ and~$g_s$, cf.~\eqref{eq:atoms}.
As such, we may invoke Lemma~\ref{lem:right-bm-locality} with~$g=g_s$,~$\tilde g=g_{T+1}$,~$\hat x=x_s$ and~$\hat u = \hat u_s$; 
this gives a relative neighborhood~$\mathcal{X}_s \subseteq \ri(\Simp)$ in which~$S_{1/2}[g_{T+1}] = S_{1/2}[g_s]$, that is~$f_{T+1} = f_s$ (cf.~\eqref{eq:fresh}--\eqref{eq:final-instance}).
\end{proof}

\begin{proposition}[Transcript consistency]
\label{prop:transcript}
The queries~$x_1,\dots,x_T$, the oracle answers returned in~\eqref{eq:oracle-answers}, and the reported point~$x_{T+1}$ coincide with the transcript of~$\Mtd$ run on the objective~$f_{T+1}$.
\end{proposition}
\begin{proof}
Fix any~$s\in [T]$. 
By Lemma~\ref{lem:local-invisibility} with~$t = T+1$, we get~$f_{T+1}=f_s$ in a relative neighborhood~$\mathcal{X}_s$ of~$x_s$; in particular~$f_{T+1}(x_s)=f_s(x_s)$. 
Moreover, since the tangent gradient of~$f \in \Fh(1)$ at~$x_s$ is determined by the values of~$f$ in a~$\Tan$-neighborhood of~$x_s$ (cf.~\eqref{eq:grad-tangent}), we have~$\bnabla f_{T+1}(x_s)=\bnabla f_s(x_s)$ and
\begin{equation}
\label{eq:oracle-equal}
\Orc_{f_{T+1}}(x_s)=\Orc_{f_s}(x_s) \qquad \forall s \in [T].
\end{equation}
We argue by induction that, when run on~$f_{T+1}$,~$\Mtd$ queries exactly~$x_1,\dots,x_T$ and reports~$x_{T+1}$. For the base case, the first query~$x_1=\Phi_1$ is fixed by~$\Mtd$ independently of the objective, cf.~\eqref{eq:method-iterates}, so it coincides with the constructed one. 
Assume, for some~$2 \le t \le T$, that the first~$t-1$ queries~$x_1,\dots,x_{t-1}$ coincide with those in the transcript of~$\Mtd$ run on~$f_{T+1}$, and the corresponding constructed answers~$\Orc_{f_{1}}(x_1), \dots, \Orc_{f_{t-1}}(x_{t-1})$ coincide with those in the transcript,~$\Orc_{f_{T+1}}(x_1), \dots, \Orc_{f_{T+1}}(x_{t-1})$. 
By~\eqref{eq:method-iterates}, the next query of~$\Mtd$ given the constructed data is~$x_{t}=\Phi_{t}(\Orc_{f_{1}}(x_1),\dots,\Orc_{f_{t-1}}(x_{t-1}))$, 
and the next query in the transcript is~$\smash{\Phi_{t}(\Orc_{f_{T+1}}(x_1),\dots,\Orc_{f_{T+1}}(x_{t-1}))}$, identical by the induction premise. 
Now due to~\eqref{eq:oracle-equal}, the next constructed answer~$\smash{\Orc_{f_{t}}(x_{t})}$ is identical to the corresponding answer~$\smash{\Orc_{f_{T+1}}(x_{t})}$ in the transcript. 
This advances the induction. 
Finally, the same argument applied for~$t = T+1$, with the report map~$\smash{\bar\Phi_{T+1}}$ in place of~$\Phi_t$, gives consistency of the reported point. 
\end{proof}

\subsection{Bounding the suboptimality gap}
\label{sec:gap-proof}

In this section, we complete the proof of Theorem~\ref{th:lower} by estimating the suboptimality gap of~$f_{T+1}$ at the reported point~$x_{T+1}$. 
To this end, we first bound~$f_{T+1}(x_{T+1})$ from below, by using Lemmas~\ref{lem:entry-growth}--\ref{lem:loss} and Proposition~\ref{prop:fresh-entries};
then we exhibit a comparator whose support includes an unexplored direction.


\proofstep{1}:~{\em Lower bound at the reported point.} 
Recall that
$
\langle x_{T+1}, e_{j_{T+1}} \rangle \le \veps_T
$
by Proposition~\ref{prop:fresh-entries}, whence 
\begin{align}
\label{eq:final-atom}
	g_{T+1}(x_{T+1})
	\ge a_{T+1}(x_{T+1})
	&= 
	-\langle x_{T+1}, e_{j_{T+1}} \rangle - T\delta_T \notag\\
	&\ge
	-\veps_T - T\delta_T.
\end{align}
Let us bound each~$g_{T+1}$-{active coordinate} (of~$\hat u_{T+1}$) at the proximal point~$\hat u_{T+1}$ of~$x_{T+1}$, i.e.~$\langle \hat u_{T+1}, e_{j_t} \rangle$ for~$t \in [T+1]$ such that~$a_t(\hat u_{T+1}) = g_{T+1}(\hat u_{T+1})$. 
For such~$t$, the final atom is dominated at~$\hat u_{T+1}$: one has
$
a_t(\hat u_{T+1}) \ge a_{T+1}(\hat u_{T+1}),
$
that is,~$- \langle \hat u_{T+1}, e_{j_t} \rangle -(t-1)\delta_T \ge -\langle \hat u_{T+1}, e_{j_{T+1}} \rangle -T\delta_T$. 
Rearranging,
\begin{equation*}
        \langle \hat u_{T+1}, e_{j_t} \rangle 
        \le \langle \hat u_{T+1}, e_{j_{T+1}} \rangle + (T+1-t)\delta_T 
        \le \langle \hat u_{T+1}, e_{j_{T+1}} \rangle + T\delta_T.
\end{equation*}
Lemma~\ref{lem:entry-growth}, applied to~$g = g_{T+1}$ and~$x = x_{T+1}$, gives~$\langle \hat u_{T+1}, e_{j_{T+1}} \rangle \le 2 \langle x_{T+1}, e_{j_{T+1}} \rangle \le 2\veps_T$, whence
\begin{equation*}
        \langle \hat u_{T+1}, e_{j_t} \rangle  \le 2\veps_T + T\delta_T
\end{equation*}
for every~$g_{T+1}$-active coordinate~$j_t$ at~$\hat u_{T+1}$. 
As such, the premise of Lemma~\ref{lem:loss} holds for~$g = g_{T+1}$ and~$x = x_{T+1}$, with~$M = 2\veps_T + T\delta_T$.
Combining that lemma (cf.~\eqref{eq:loss-bound}) with~\eqref{eq:final-atom} results in
\begin{equation}
\label{eq:env-lower}
\begin{aligned}
\Env_{1/2}[g_{T+1}](x_{T+1})
        \stackrel{\eqref{eq:loss-bound}}{\ge}
g_{T+1}(x_{T+1}) - \frac{M}{2}
        \stackrel{\eqref{eq:final-atom}}{\ge}
         -2\veps_T - \frac{3T}{2} \delta_T.
\end{aligned}
\end{equation}

\proofstep{2}:~{\em Upper bound at the comparator}. 
Consider~$\bar x = \frac{1}{T+1} \sum_{t \in [T+1]} e_{j_{t}}$, i.e.~uniform on the selected coordinates; note that this includes the coordinate~$j_{T+1}$ revealed {\em after}~$\Mtd$'s reported point~$x_{T+1}$. 
Clearly,~$\bar x \in \Simp$. 
At~$\bar x$, each atom evaluates as~$a_t(\bar x) = -\langle \bar x, e_{j_t}\rangle-(t-1)\delta_T = -\tfrac1{T+1}-(t-1)\delta_T$. 
Since~$(t-1)\delta_T \ge 0$,  we get
\begin{equation*}
        g_{T+1}(\bar x)
        =
        \max_{t \in [T+1]}  a_t(\bar x)
        \le-\frac1{T+1}.
\end{equation*}
Since~$\bar x$ is feasible in~\eqref{eq:envelope}, this bound extends to the envelope:
$
        \Env_{1/2}[g_{T+1}](\bar x) \le g_{T+1}(\bar x) = -\frac1{T+1}.
$
Combining this with~\eqref{eq:env-lower}, we get 
\begin{align*}
\Env_{1/2}[g_{T+1}](x_{T+1})-\min_{x \in \Simp}\Env_{1/2}[g_{T+1}](x)
&\ge
\frac1{T+1}-2\veps_T-\frac{3T}{2}\delta_T.
\end{align*}
Plugging in the parameter values~$ \veps_T :=\frac{1}{8(T+1)^2}$ and $ \delta_T = \frac{1}{3(T+1)^2} $ from~\eqref{eq:params}, and using~\eqref{eq:final-instance}, we get 
\begin{equation*}
\begin{aligned}
        f(x_{T+1})-\min_{x \in \Simp}f(x)
        \ge
        \frac{1}{2}\left( \frac1{T+1}-2\veps_T-\frac{3T}{2}\delta_T \right)
        &= 
         \frac{1}{2(T+1)}\left(1-\frac{1}{4(T+1)}-\frac{T}{2(T+1)} \right) \\
        &=
         \frac{2T+3}{8(T+1)^2}
         > \frac{1}{4(T+1)}.
\end{aligned}
\end{equation*}
This completes the proof in the case of~$L = 1$. Recall that the general case follows by homogeneity.
%
\qed

\section{Quantum extension}
\label{sec:quantum}


Let~$\Herm$ be the vector space of Hermitian~$d \times d$ matrices with inner product~$\inner{U}{V}:=\tr(UV)$, and let~$\Herm_+$ be the positive-semidefinite cone in~$\Herm$. 
Define the spectrahedron of density matrices
\begin{equation}
\label{eq:dens-def}
\Dens:=\{X\in\Herm_+:  \tr(X) = 1\}.
\end{equation}
The set~$\smash{\ri(\Dens)}$ consists of positive-definite unit-trace matrices in~$\smash{\Herm}$.
To formulate the result, it is convenient to use the notion of spectral functions.
Recall that any symmetric function~$\smash{\phi: \R^d \to \R}$ defines the spectral function~$\phi \circ \lambda$ on~$\Herm$,
where~$\lambda(X) = (\lambda_{1}(X),\dots,\lambda_d(X))$ is the ordered spectrum of~$X$.
It is well-known (e.g.~\cite[Thm.~1.1]{lewis1996derivatives}) that the gradient of~$\phi \circ \lambda$ at~$X = Q \Diag(\lambda(X)) Q^\dagger$ reads
\[
\nabla [\phi \circ \lambda](X) = Q \Diag(\nabla \phi(\lambda(X))) Q^\dagger
\] 
where~$A^\dagger$ is the conjugate transpose of a matrix~$A$ and~$\Diag(\cdot)$ maps~$x \in \R^d$ to the diagonal~$d \times d$ matrix with~$x$ on the diagonal; 
in other words,~$\smash{\nabla [\phi \circ \lambda](X)}$ is a Hermitian matrix with the same eigenbasis as~$\smash{X}$, whose spectrum is given by the gradient of~$\phi$ evaluated at the spectrum of~$X$.
In particular, negative entropy~\eqref{eq:intro-entropy} gives the negative von Neumann entropy~$h(\lambda(X))=\tr(X\log X)$. 
It is also well-known (e.g.~\cite{baes2007convexity}) that strict convexity of~$\phi \circ \lambda$ is equivalent to that of~$\phi$; this allows to extend Bregman divergences to~$\Herm$ via
$D_{\phi \hspace{1pt}\circ \lambda}(X,U) := \phi(\lambda(X)) - \phi(\lambda(U)) - \langle \nabla [\phi \circ \lambda](U), X-U \rangle$. 
In particular, {\em (Umegaki's) quantum relative entropy} between~$X \in \Dens$ and~$U \in \ri(\Dens)$, as given by
\begin{equation}
\label{eq:qre-def}
\qre{X}{U}:=\tr(X(\log X-\log U)),
\end{equation}
extends KL divergence in the above sense~\cite{hiai1991proper} and can be further extended to~$\Dens \times \Dens$ as in~\eqref{eq:intro-kl}.
As such, the appropriate generalization of~$\Fh(L)$ is the class~$\Qd(L)$ of convex functions~$F: \Dens \to \R$ continuously differentiable on~$\ri(\Dens)$ and satisfying the following inequalities:
\begin{equation}
\label{eq:quantum-sandwich}
0\le F(X)-F(U)-\inner{\bnabla F(U)}{X-U} 
\le L\qre{X}{U} \qquad \forall (X,U) \in \Dens \times \ri(\Dens).
\end{equation}
Here, as in~\eqref{eq:intro-sandwich}, we let~$\bnabla F$ be the tangent gradient of~$F$ on~$\ri(\Dens)$, so that~$\bnabla F(U) \in \spanof{I_d}^\perp$. 

We can now state a noncommutative generalization of Theorem~\ref{th:intro-lower}.


\begin{theorem}
\label{th:quantum}
Let~$d,L,T$ be as in the premise of Theorem~\ref{th:lower}. 
For every deterministic method that makes~$T$ queries~$X_1, \dots, X_T \in \ri(\Dens)$ of the oracle~$X \mapsto (F(X),\bnabla F(X))$ and returns~$X_{T+1}\in\Dens$, there exists~$F(\cdot)\in\Qd(L)$ such that
\begin{equation}
\label{eq:quantum-bound}
F(X_{T+1})-\min_{X\in\Dens}F(X)>
\frac{L}{4(T+1)}.
\end{equation}
\end{theorem}

\begin{proof}[Proof sketch]
\proofstep{1}. First of all, we observe that the function~$F: \Dens \to \R$ given by~$F(X) = f(\diag(X))$, where~$f$ is the hard instance~$f \in \Ent(L)$ from Theorem~\ref{th:lower} and~$\diag(X) = (X_{11}, \dots, X_{dd})$, belongs to the class~$\Qd(L)$. 
Indeed, convexity of~$F$ follows from that of~$f$, since~$\diag(\cdot)$ is a linear mapping.
On the other hand, defining the diagonal truncation operator~$\Pin(H) := \Diag(\diag(H))$, we have
\begin{equation}
\label{eq:dpi}
\Div(\diag(X),\diag(U))
= \qre{\Pin(X)}{\Pin(U)}
\le\qre{X}{U} \qquad \forall (X,U) \in \Dens \times \ri(\Dens).
\end{equation}
Here, the identity holds by~\eqref{eq:qre-def} and~\eqref{eq:intro-kl}; the final estimate is by Lindblad's data processing inequality~\cite{Lindblad1975,Uhlmann1977}.  
Meanwhile, by the chain rule~$\bnabla F(U)$ is a diagonal matrix with~$\bnabla f(\diag(U))$ on the main diagonal; as a result, we have the following in terms of~$x = \diag(X)$ and~$u = \diag(U)$,
\begin{align*}
F(X)-F(U)-\inner{\bnabla F(U)}{X-U} 
=
f(x)-f(u)-\inner{\bnabla f(u)}{x-u}  
\le L\Div(x,u)
\le L\qre{X}{U},
\end{align*}
as required in~\eqref{eq:quantum-sandwich}. 
As such,~$F(\cdot) = f(\diag(\cdot))$ indeed belongs to~$\Qd(L)$ whenever~$f \in \Fh(L)$. 

\proofstep{2}. To complete the proof, we note that a method~$\tilde \Mtd$ with queries~$X_t \mapsto (F(X_t),\bnabla F(X_t))$, when run on~$F(\cdot) = f(\diag(\cdot))$, receives answers~$(f(x_t),\Diag(\bnabla f(x_t)))$ where~$x_t = \diag(X_t)$. 
Since the matrix~$\Diag(\bnabla f(x_t))$ contains the same information as~$\bnabla f(x_t)$, the sequence
$(x_t,f(x_t),\bnabla f(x_t))_{t \in [T]}$ 
is the transcript of {\em some}~$\smash{\Mtd \in \FOM_d(T)}$ run on~$\smash{f \in \Fh(L)}$ and returning~$x_{T+1}$.
It remains to pick~$f$ as in Theorem~\ref{th:lower} and combine~\eqref{eq:intro-lower} with the fact that~$\smash{\min_{X \in \Dens} f(\diag(X)) = \min_{x \in \Simp} f(x)}$.
\end{proof}
In Appendix~\ref{app:quantum-proof} we give an explicit minimax formulation of Theorem~\ref{th:quantum} analogous to Theorem~\ref{th:lower}, together with the full presentation of the emulation argument in step \proofstep{2} of the above proof sketch.

%
\section{Pointwise minimal interpolant}
\label{sec:outro}

A promising approach towards improving Theorem~\ref{th:intro-lower} would adapt the primal-dual estimate function (PDEF) framework of Florea and Nesterov~\cite{florea2025optimal}
to the class~$\Fh(L)$. 
This framework is built around the {\em optimal interpolating lower model}, defined as the pointwise-minimal function in the class of convex functions with~$L$-Lipschitz gradient (w.r.t.~$\|\cdot\|_2$), that interpolates the given data~$\smash{\{(x_s,f_s,\xi_s)\}_{s \in [t]}}$. The idea is that such a model can be updated incrementally: the next primal point~$x_{t+1}$ is selected by minimizing the current model~$\hat f_t$; a resisting oracle then selects an answer~$(f_{t+1}, \xi_{t+1})$ compatible with~$\hat f_t$, and the model is updated to incorporate the new point. 
Here, we make the first step towards implementing this program, by deriving an explicit form of the optimal interpolating lower model for the class~$\Fh(L)$, as a consequence of the results of Dragomir~\cite[Chap.~5]{DragomirThesis2021}. 
We note that the remaining step
would be to construct a resisting oracle that gives the tightest final lower bound.

Given the data~$\cI=\{(x_i,f_i,\xi_i)\}_{i \in [N]}$ with~$(x_i,f_i,\xi_i) \in \ri(\Simp) \times \R \times \Tan$, we define the quantities
\begin{equation}
\label{eq:interpol-quantities}
r_i := f_i - \langle \xi_i, x_i \rangle,
\qquad
w_i := \exp({L^{-1}}{(r_i \ones + \xi_i)}), 
\qquad
W_{\cI} := \conv\{w_1, \dots, w_N\}.
\end{equation}
(In the sequel,~$\exp(\cdot)$,~$\log(\cdot)$ and division are entrywise.)
The next lemma, derived from~\cite[Thm.~5.11]{DragomirThesis2021} and proved in Appendix~\ref{app:interpol}, provides explicit interpolation conditions for the class~$\Fh(L)$.

\renewcommand{\ip}[2]{\langle #1, #2 \rangle}

\begin{lemma}
\label{lem:entropy-interpolation}
The existence of~$f \in \Fh(L)$ that satisfies~$f(x_i) = f_i$ and~$\bnabla f(x_i) = \xi_i$ for all~$i \in [N]$ is equivalent to the following inequalities in terms of~\eqref{eq:interpol-quantities}:
\begin{equation}
\label{eq:entropy-interpolation}
 \ip{x_i}{w_j/w_i}\le 1 \qquad \forall i,j \in [N].
\end{equation}
\end{lemma}


Using these interpolation conditions, we now derive the pointwise minimal interpolant in~$\Fh(L)$.

\begin{proposition}
\label{prop:log-hull}
Assume~$\cI = \{(x_i,f_i,\xi_i)\}_{i \in [N]}$ satisfies~\eqref{eq:entropy-interpolation}.
The function~$f_{\cI}: \Simp \to \R$ defined by
\[
 f_{\cI}(x):=L\max_{w\in W_{\cI}}\ip{x}{\log w}   
\]
belongs to~$\Fh(L)$, interpolates~$\cI$, 
and is pointwise minimal among all functions with these properties.
\end{proposition}

\begin{proof}
\proofstep{1}. Let us verify that~$f_{\cI} \in \Fh(L)$. 
The function~$f_{\cI}$ is convex, and one has
\[
Lh(x)-f_{\cI}(x)=L\min_{w\in W_{\cI}} \sum_{k \in [d]} (x)_k \log\frac{(x)_k}{(w)_k}.
\]
The summand is jointly convex in~$((x)_k,(w)_k)$ as the perspective of a convex function~$-\log(\cdot)$, so~$Lh-f_{\cI}$ is convex. 
For~$x\in\ri(\Simp)$, strict concavity of~$w\mapsto\ip{x}{\log w}$ on~$W_{\cI}$ gives a unique maximizer~$w(x)$ that continuously depends on~$x$. Whence by  Danskin's theorem,~$f_{\cI} \in C^1(\ri(\Simp))$ with
$
\bnabla f_{\cI}(x)=L\ProjTan (\log w(x)),
$
and therefore~$f_{\cI}\in\Fh(L)$.

\proofstep{2}. 
We show that~$f_{\cI}$ interpolates~$\cI$. 
For all~$w \in W_{\cI}$, we have~$\ip{x_i}{w/w_i} \le 1$  by~\eqref{eq:entropy-interpolation}, whence by concavity of~$\log(\cdot)$,
\[
 \ip{x_i}{\log(w/w_i)}\le\log\ip{x_i}{w/w_i} \le 0.
\]
Thus~$w(x_i) = w_i$, therefore
$f_{\cI}(x_i)=L\ip{x_i}{\log w_i}=f_i$
and
$
\bnabla f_{\cI}(x_i)=L\ProjTan(\log w_i)=\xi_i,
$
cf.~\eqref{eq:interpol-quantities}.

\proofstep{3}.
Finally, consider arbitrary~$F\in\Fh(L)$ that interpolates~$\cI$, i.e.~$F(x_i)=f_i$ and~$\bnabla F(x_i)=\xi_i$ for all~$i\in[N]$. 
Fix~$x\in\ri(\Simp)$, and set~$\xi:=\bnabla F(x)$ and~$r:=F(x)-\ip{\xi}{x}$. Applying Lemma~\ref{lem:entropy-interpolation} to the interpolation data~$\cI \cup \{(x,F(x),\xi)\}$ we get
\begin{equation}
\label{eq:interpol-condition-old-to-new}
 \exp({L^{-1}}{r_i})
 \ip{x}{\exp({L^{-1}}{(\xi_i-\xi)})}
 \le \exp({L^{-1}}{r})
 \qquad \forall i \in[N]
\end{equation}
where~$r = F-\langle\xi,x\rangle$.
As a result, for~$w=\sum_{i \in [N]} \lambda_i w_i\in W_{\cI}$ we get,  by using the concavity of~$\log(\cdot)$, 
\begin{align*}
L\ip{x}{\log w}
 &\le \ip{\xi}{x}
 +L\log\Bigg(\sum_{i \in [N]} \lambda_i \exp({L^{-1}}{r_i})
 \ip{x}{\exp({L^{-1}}{(\xi_i-\xi)})}\Bigg)
\stackrel{\eqref{eq:interpol-condition-old-to-new}}{\le} \ip{\xi}{x}+r 
 = F(x).
\end{align*}
Maximizing over~$w$ we get~$f_{\cI}(x)\le F(x)$ on~$\ri(\Simp)$. 
For~$\bar x$ on the relative boundary of~$\Simp$, consider the segment~$x(t) = (1-t)\bar x + t y$ with~$y \in \ri(\Simp)$. Then by continuity of~$f_{\cI}$ and convexity of~$F$ we get~$f_{\cI}(\bar x)=\lim_{t \downarrow 0}f_{\cI}(x(t))\le\limsup_{t\downarrow0}F(x(t))\le \limsup_{t \downarrow 0} (1-t) F(\bar x) + t F(y) = F(\bar x)$.
\end{proof}
\section*{Acknowledgments}

J.~M.~Aguirre is supported by the NSF Graduate Research Fellowship under Grant No.~DGE-2039655.
D.~M.~Ostrovskii thanks Radu-Alexandru Dragomir, Adrien Taylor, and Alexandre d'Aspremont for interesting discussions pertaining to this problem.
\appendix
\section{Objective extension to~$\R^d_+$ with full-gradient oracle access}
\label{app:conic}

\paragraph{Tangent gradients and~$1$-homogeneous extension over~$\R^d_+$.}
Recall that any function~$f$ defined on~$\Simp$ and differentiable in~$\ri(\Simp)$---including the hard instances presented in Section~\ref{sec:hard-subclass}---admits the~1-homogeneous extension on~$\R^d_+$ (see, e.g.,~\cite{ovcharov2018proper}), defined as follows: setting~$s(x) := \ip{\ones}{x}$,
\begin{equation}
\label{eq:conic-extension}
f^+(x) = s(x) f (s(x)^{-1}x).
\end{equation}
This function is differentiable in~$\R^d_{++}$. 
An explicit calculation gives its full gradient for~$x \in \ri(\Simp)$:
\begin{equation}
\label{eq:grad-conic}
\nabla_\star f^+(x) = \bnabla f(x) + (f(x) - \langle \bnabla f(x), x \rangle) \ones.
\end{equation}
Thus, the oracle~$\Orc_f(x) = (f(x),\bnabla f(x))$ is at least as strong as the full gradient oracle~$x \mapsto \nabla_\star f^+(x)$ in the extended class~$\{\smash{f^+: f \in \Fh(L)}\}$, in the sense that one can emulate the latter by using the former's answer at the same query point.
Intuitively, this should imply that access to the full gradient of~$\smash{f^+}$ {\em cannot} lead to a faster convergence rate than access to the oracle~$\smash{\Orc_f}$, as~$\smash{\Orc_f}$ gives at least as much information as~$\nabla_\star f^+$.
The result we shall present next formalizes this comparison and leverages Theorem~\ref{th:lower} to obtain a lower bound for the class of functions on the positive orthant, convex and smooth relative to the unnormalized negative entropy
$
\smash{h_\e(x) := \sum_{k \in [d]} (x)_k \log(x)_k-(x)_k}
$
on~$\R^d_+$. 

Define~$\smash{\Ent_{d}^+(L)}$ as the class of functions~$\smash{\tilde f: \R^d_+ \to \R}$ that are convex,~$\smash{C^1(\R^d_{++})}$, and such that 
\begin{equation}
\label{eq:sandwich-conic}
\tilde f(x)- \tilde f(u)-\inner{\nabla_\star \tilde f(u)}{x-u} \le LD_{h_\e}(x,u) \quad \forall (x,u) \in \R^d_+ \times  \R^d_{++}.
\end{equation}
Since~$\smash{\inner{\nabla_\star \tilde f(x)}{v} = \inner{\bnabla \tilde f(x)}{v}}$ for all tangent directions~$v \in \Tan$, and~$h_e(x) = h(x)-1$ for all~$x \in \Simp$, the restriction of~$\tilde f \in \smash{\Fh^+(L)}$ to~$\Simp$ belongs to~$\Fh(L)$. The next lemma gives a partial converse.
\begin{lemma}
\label{lem:conic-extension}
The 1-homogeneous extension $f^+(x)$ of any~$f \in \Fh(L)$, cf.~\eqref{eq:conic-extension}, belongs to~$\Fh^+(L)$.
\end{lemma}
\begin{proof}
$f^+$ is convex as the composition of the perspective transformation of~$f$ and a linear mapping. 
Clearly,~$f^+ \in C^1(\R^d_{++})$. 
Finally, by writing 
\[
Lh_\e(x)-f^+(x)
 =s(x) (Lh(y)-f(y))\big|_{y = s(x)^{-1}x} +Ls(x)\log s(x)-L s(x)
\]
we see that~$Lh_\e - f^+$ is convex, which is equivalent to~\eqref{eq:sandwich-conic}. 
Indeed, the first term is convex as the composition of the perspective transformation of~$Lh-f$ (which is convex) and a linear mapping.
\end{proof}
%

Similarly, we can extend the class of first-order methods, by considering collections of mappings
\begin{equation}
\label{eq:method-maps-conic}
\begin{aligned}
\Psi_t: \R^{d(t-1)}  &\to \ri(\Simp), \qquad t \in [T];\\
\bar  \Psi_{T+1}: \R^{dT} &\to \Simp,
\end{aligned}
\end{equation}
where~$\R^d$ replaces the oracle answer space~$\Omega = \R \times \Tan$ in~\eqref{eq:method-maps}; 
this corresponds to minimizing~$\tilde f$ over~$\Simp$.
Let~$\FOM_d^+(T)$ be the class of all methods defined by~\eqref{eq:method-maps-conic} and define the minimax risk
\begin{equation}
\label{eq:minimax-risk-conic}
\Risk_{d}^+(T,L) \;\; 
:=
\inf_{\tilde\Mtd^{\vphantom{{{2^2}^2}}} \,\in\, \FOM_{d}^+(T)} 
\quad
\sup_{\tilde f \,\in\, \Fh^+(L)} 
\quad
\left\{ \tilde f(x_{T+1}^{\tilde \Mtd}(\tilde f))- \min_{x \in \Simp} \tilde f(x) \right\},
\end{equation}
where the iterate sequence is generated by running~$\tilde \Mtd$ with the full gradient oracle~$x \mapsto \nabla \tilde f(x)$, i.e.
\begin{equation}
\label{eq:method-iterates-conic}
x_1^{\tilde\Mtd} = \Psi_1^{\vphantom{\tilde\Mtd}}, 
\;\;
x_2^{\tilde\Mtd}(\tilde f) = \Psi_2^{\vphantom{\tilde\Mtd}}(\nabla \tilde f(x_1^{\tilde\Mtd})), 
\;\; \ldots, \;\; 
x_{T+1}^{\tilde\Mtd}(\tilde f) = \bar\Psi_{T+1}^{\vphantom{\tilde\Mtd}}(\,\nabla \tilde f(x_1^{\tilde\Mtd}),\; \dots,\; 
\nabla \tilde f(x_{T}^{\tilde\Mtd}(\tilde f))).
\end{equation}
%
%
\begin{proposition}
\label{prop:emulation}
The minimax risks defined in~\eqref{eq:minimax-risk-conic} and~\eqref{eq:minimax-risk} satisfy
$
\Risk_{d}^+(T,L) \ge \Risk_{d}(T,L).
$
\end{proposition}

\begin{proof}
Replacing~$\Ent_d^+(L)$ with the smaller (due to Lemma~\ref{lem:conic-extension}) class~$\{f^+: f \in \Fh(L)\}$, we get
\[
\begin{aligned}
\Risk_{d}^+(T,L) 
&\ge 
\inf_{\tilde\Mtd^{\vphantom{{{2^2}^2}}} \,\in\, \FOM_{d}^+(T)} 
\quad
\sup_{f \,\in\, \Fh(L) \vphantom{\wt \Fh(L)}} 
\quad
\left\{  f^+(x_{T+1}^{\tilde \Mtd}(f^+))- \min_{x \in \Simp} f^+(x) \right\} \\ 
&= 
\inf_{\tilde\Mtd^{\vphantom{{{2^2}^2}}} \,\in\, \FOM_{d}^+(T)} 
\quad
\sup_{f \,\in\, \Fh(L) \vphantom{\wt \Fh(L)}} 
\quad
\left\{  f(x_{T+1}^{\tilde \Mtd}(f^+))- \min_{x \in \Simp} f(x) \right\}
\ge \Risk_{d}(T,L).
\end{aligned}
\]
Here the identity holds since the mappings in~\eqref{eq:method-maps-conic} output points in~$\Simp$ (where~$f^+ = f$). 
For the last inequality, we use~\eqref{eq:grad-conic} and observe
that the execution on~$f^+$ of~$\smash{\tilde\Mtd \in \FOM_d^+(T)}$, as per~\eqref{eq:method-maps-conic},  produces the same query sequence
as the execution on~$f$ of the method
$
\Mtd \in \FOM_d(T),
$
defined by
\newcommand{\bg}{\xi}
\[
\begin{aligned}
\Phi_1 &:= \Psi_1, \\
\Phi_2(v_1, \bg_1) 
&:= \Psi_2 \left(\bg_1 + \big(v_1 -\langle \bg_1,  \Phi_1 \rangle \big) \ones \right), \\
\Phi_3(v_1,\bg_1, v_2, \bg_2) 
&:= \Psi_3 
		\left(
			\bg_1 + \big(v_1 -\langle \bg_1,  \Phi_1 \rangle \big) \ones,  \;\;
			\bg_2 + \big(v_2 -\langle \bg_2,  \Phi_2(v_1,\bg_1) \rangle \big) \ones
		\right), \\
&\;\;\vdots \vspace{-0.1cm} \\ \vspace{-0.1cm}
\bar\Phi_{T+1}(v_1,\bg_1, \dots, v_T, \bg_T) \\
:= \bar\Psi_{T+1} 
		\big(
		\bg_1 + \big(v_1 - &\langle \bg_1, \Phi_1 \rangle \big) \ones, \;\dots,\;
		\bg_T + \big(v_T -\langle \bg_T,  \Phi_T(v_1,\bg_1, \dots, v_{T-1},\bg_{T-1}) \rangle \big) \ones
		\big).
\end{aligned}
\]
This can be verified by induction over~$t$, in the same way as in the proof of Proposition~\ref{prop:transcript}.
This fact implies the desired inequality, since in~$\Risk_{d}(T,L)$ the infimum is over {\em all} methods in~$\FOM_d(T)$. 
\end{proof}
\section{Explicit formulation and emulation argument for Theorem~\ref{th:quantum}}
\label{app:quantum-proof}

We first adapt the definition of first-order methods to make them in line with the class~$\Qd(L)$, cf.~\eqref{eq:quantum-sandwich}. 
Namely, we define the methods in~$\smash{\FOM_d^{\mathsf{H}}(T)}$  as collections of mappings (cf.~\eqref{eq:method-maps}):
\begin{equation}
\label{eq:method-maps-quantum}
\begin{aligned}
\boldsymbol{\Phi_t}: \boldsymbol{\Omega}^{t-1}  &\to \ri(\Dens), \qquad t \in [T]; \\
\boldsymbol{\bar \Phi_{T+1}}: \boldsymbol{\Omega}^T &\to \Dens,
\end{aligned}
\end{equation}
where~$\bOmega = \R \times \spanof{I_d}^\perp$ and~$\bOmega^0$ is a singleton.  
A given method in~$\FOM_d^{\mathsf{H}}(T)$ sequentially queries
\begin{equation}
\label{eq:oracle-quantum}
\Orc_F^{\mathsf{H}}(X) 
:= 
(F(X),\bnabla F(X))
\end{equation}
and uses the responses to form the sequence (cf.~\eqref{eq:method-iterates})
\begin{equation}
\label{eq:method-iterates-quantum}
X_1^{\vphantom\Mtd} = \bPhi_1^{\vphantom\Mtd}, 
\;\;
X_2^{\vphantom \Mtd}(F) = \bPhi_2^{\vphantom\Mtd}(\Orc_F^{\mathsf{H}}(X_1^{\vphantom\Mtd})), 
\;\; \ldots, \;\; 
X_{T+1}^{\vphantom\Mtd}(F) = \bar\bPhi_{T+1}^{\vphantom\Mtd}(\,\Orc_F^{\mathsf{H}}(X_1^{\vphantom\Mtd}),\; \dots,\; \Orc_F^{\mathsf{H}}(X_{T}^{\vphantom \Mtd}(F))).
\end{equation}
In these terms, Theorem~\ref{th:quantum} claims that~$\Risk_d(T,L)$, cf.~\eqref{eq:minimax-risk}, is a lower bound for the quantum risk
\begin{equation}
\label{eq:minimax-risk-quantum}
\Risk_{d}^{\mathsf{H}}(T,L) \; 
:=
\inf_{{\bPhi_1, \dots, \bPhi_T, \bar\bPhi_{T+1}}^{\vphantom{\mathbf{H}}}} 
\quad
\sup_{F \,\in\, \Fh^{\mathsf{H}}(L)} 
\quad
\left\{ F(X_{T+1}(F))- \min_{X \in \Dens} F(X) \right\},
\end{equation}
where the infimum is over~$\FOM_d^{\mathsf{H}}(T)$, cf.~\eqref{eq:method-maps-quantum}--\eqref{eq:method-iterates-quantum}. 
By the argument in step~(\proofstep{1}) of the proof sketch of Theorem~\ref{th:quantum}, we have~$(f \circ \diag) \in \Fh^{\mathsf{H}}(L)$, whence
\begin{align}
\Risk_{d}^{\mathsf{H}}(T,L) \; 
&\ge 
\inf_{{\bPhi_1, \dots, \bPhi_T, \bar\bPhi_{T+1}}^{\vphantom{\mathsf{H}}}} 
\quad
\sup_{f \,\in\, \Fh(L)} 
\quad
\left\{ f(\diag(X_{T+1}(f \circ \diag)))- \min_{X \in \Dens} f(\diag(X)) \right\} \notag\\
&=
\inf_{{\bPhi_1, \dots, \bPhi_T, \bar\bPhi_{T+1}}^{\vphantom{\mathsf{H}}}} 
\quad
\sup_{f \,\in\, \Fh(L)} 
\quad
\left\{ f(\diag(X_{T+1}(f \circ \diag)))- \min_{x \in \Simp} f(x) \right\}\,.
\label{eq:quantum-risk-chain}
\end{align}
Combining~$\bnabla [f \circ \diag](X) =  \Diag(\bnabla f(\diag(X)))$ with~\eqref{eq:oracle-quantum}--\eqref{eq:method-iterates-quantum}, for~$X_t = X_t(f \circ \diag)$  we get
\[
\begin{aligned}
X_{t} 
&= \bPhi_t(f(\diag(X_1)), \Diag(\bnabla f(\diag(X_1))), \dots, f(\diag(X_{t-1})), \Diag(\bnabla f(\diag(X_{t-1}))))
\end{aligned}
\]
when~$2 \le t \le T$, and
\[
\begin{aligned}
X_{T+1} &= \bar\bPhi_{T+1}(f(\diag(X_1)), \Diag(\bnabla f(\diag(X_1))), \dots, f(\diag(X_{T})), \Diag(\bnabla f(\diag(X_{T})))).
\end{aligned}
\]
Since in both cases~$X_t$ enters the right-hand side  only through the diagonal~$x_t := \diag(X_t)$, we get
\[
\begin{aligned}
x_{t}(f) 
&= \Phi_t(f(x_1), \bnabla f(x_1), \dots, f(x_{t-1}(f)), \bnabla f(x_{t-1}(f))), \quad 2 \le t \le T; \\
x_{T+1}(f) 
&= \bar\Phi_{T+1}(f(x_1), \bnabla f(x_1), \dots, f(x_{T}(f)), \bnabla f(x_{T}(f))),
\end{aligned}
\]
with~$\Phi_t: \Omega^{t-1} \to \ri(\Simp)$ defined by~$\Phi_t(v_1, g_1,\dots, v_{t-1}, g_{t-1}) = \bPhi_t(v_1, \Diag(g_1),\dots, v_{t-1}, \Diag(g_{t-1}))$; ditto for~$\bar\Phi_{T+1}$.  
These maps define a method in~$\FOM_d(T)$ with transcript~$(x_t,f(x_t),\bnabla f(x_t))_{t \in [T]}$ consistent with~$f \in \Fh(L)$;  plugging the reported point~$\diag(X_{T+1}(f \circ \diag)) = x_{T+1}(f)$ in~\eqref{eq:quantum-risk-chain},
\[
\Risk_{d}^{\mathsf{H}}(T,L) \; 
\ge 
\inf_{{\Phi_1, \dots, \Phi_T, \bar\Phi_{T+1}}^{\vphantom{\mathbf{H}}}} 
\quad
\sup_{f \,\in\, \Fh(L)} 
\quad
\left\{ f(x_{T+1}(f))- \min_{x \in \Simp} f(x) \right\} 
= \Risk_d(T,L). \qquad \qed
\]

\section{Proof of Lemma~\ref{lem:entropy-interpolation}} 
\label{app:interpol}

The gradient and convex conjugate of~$h_\e$ are~$\nabla_{\star} h_\e(x)=\log x$ and $h_\e^*(\xi)= \langle \ones, \exp(\xi) \rangle$, respectively. 

Lift~$\cI = \{(x_i,f_i,\xi_i)\}_{i \in [N]}$ to $\tilde\cI := \{(x_i,f_i, \tilde \xi_i)\}_{i \in [N]}$, where~$\tilde \xi_i:= \xi_i + r_i\ones$; cf.~\eqref{eq:interpol-quantities}. 
By~\eqref{eq:grad-conic}, we have~$(f(x_i),\nabla f(x_i))= (f_i,\xi_i)$ if and only if~$(f^+(x_i),\nabla_{\star} f^+(x_i))=(f_i,\tilde \xi_i)$. 
Meanwhile,~\cite[Thm.~5.11]{DragomirThesis2021} gives equivalent conditions of the existence of~$\tilde f \in \Fh^+(L)$ interpolating the data~$\tilde\cI$: for all~$i,j \in [N],$
\begin{equation}
\label{eq:interpol-conditions-unnormalized}
 A_{ij}:=f_i-f_j-\ip{\tilde \xi_j}{x_i-x_j}
 \ge L D_{h_\e^*}\big(\log x_i-{L^{-1}}{(\tilde \xi_i-\tilde \xi_j)},\log x_i\big).
\end{equation}
Since~$x_i,x_j\in\Simp$, we have~$A_{ij}=f_i-f_j-\ip{\xi_j}{x_i-x_j}$. Direct expansion of the right-hand side gives
\begin{align*}
D_{h_\e^*}(\log x_i+{L^{-1}}{(\tilde \xi_j-\tilde \xi_i)},\log x_i)
 &=\exp({L^{-1}}{(r_j-r_i)})
	\ip{x_i}{\exp({L^{-1}}{(\xi_j-\xi_i)})}
    -\ip{x_i}{L^{-1}(\tilde \xi_j-\tilde \xi_i)}-1\\
 &=L^{-1}A_{ij}
 	+\exp({L^{-1}}{(r_j-r_i)})
   	\ip{x_i}{\exp({L^{-1}}{(\xi_j-\xi_i})}-1.
\end{align*}
Thus~\eqref{eq:interpol-conditions-unnormalized} is equivalent to
$
\exp({L^{-1}}{(r_j-r_i)})
\ip{x_i}{\exp({L^{-1}}{(\xi_j-\xi_i})} \le 1,
$
that is to~\eqref{eq:entropy-interpolation}; cf.~\eqref{eq:interpol-quantities}.

Let us draw conclusions. 
If~$f \in \Fh(L)$ interpolates~$\cI$, then there exists a function in~$\Fh^+(L)$, namely the 1-homogeneous extension~$f^+$ of~$f$, that interpolates~$\tilde \cI$, so~\eqref{eq:entropy-interpolation} holds.
Conversely,~\eqref{eq:entropy-interpolation} implies the existence of~$\smash{\tilde f \in \Fh^+(L)}$ that interpolates~$\tilde \cI$, but then the restriction~$f$ of~$\smash{\tilde f}$ to~$\Simp$ interpolates~$\cI$: indeed,~$\tilde f(x_i) = f(x_i)$ since~$x_i \in \Simp$, and~$\bnabla f(x_i) = \ProjTan(\nabla_{\star} \tilde f(x_i)) = \ProjTan(\tilde \xi_i) = \xi_i$. 
\qed

\bibliographystyle{plain}
\bibliography{references}

\end{document}